\documentclass[12pt]{amsart}

\usepackage{amsfonts, amsthm, amsmath, amssymb}
\usepackage{hyperref}
\hypersetup{colorlinks=false}

\usepackage[margin=1.25in]{geometry}

\usepackage{helvet}

\newtheorem{theorem}{Theorem}
\newtheorem{Lemma}{Lemma}
\newtheorem{proposition}{Proposition}

\theoremstyle{definition}

\newtheorem{remark}{Remark}

\renewcommand{\d}{\mathrm{d}}
\renewcommand{\phi}{\varphi}
\renewcommand{\rho}{\varrho}

\title[Sub-Weyl bounds for $GL(2)$ $L$-functions]{Sub-Weyl bounds for $GL(2)$ $L$-functions}

\author{Ritabrata\ Munshi}
\address{School of Mathematics\\ 
Tata Institute of Fundamental Research\\
1 Homi Bhabha Road\\ Colaba\\Mumbai 400005\\ India}
\curraddr{Statistics and Mathematics Unit\\ Indian Statistical Institute\\ 203 B.T. Road\\ Kolkata 700108\\ India} 
\email{rmunshi@math.tifr.res.in, ritabratamunshi@gmail.com}


\begin{document}

\begin{abstract}
We obtain a sub-Weyl bound in the $t$-aspect for $L(1/2+it,F)$ where $F$ is a Hecke modular form.
\end{abstract}

\maketitle


\section{Introduction}
\label{introd}

In 1982, A. Good \cite{Go} extended the classic bound of Weyl, Hardy and Littlewood (\cite{Lan}, \cite{We}) to degree two $L$-functions. More precisely, for holomorphic Hecke cusp forms $F$, he proved that the associated $L$-functions satisfy
$$
L(1/2+it,F)\ll t^{1/3+\varepsilon}
$$ 
on the central line. A simpler proof of this result was given  by Jutila \cite{Ju}, and it was generalised to the case of Maass wave forms by Meurman \cite{Me}. (The case of general level has been addressed recently in \cite{BMN}.) Though in the last hundred years the bound on the Riemann zeta function has been improved, albeit mildly (Bourgain's recent work  \cite{Bou} gives the exponent $1/6-1/84$), Good's bound has so far remained unsurpassed. Indeed sub-Weyl bounds have been rarely achieved in the long history of subconvexity for degree one $L$-functions (see \cite{Mil} for another instance). But the Weyl bound has never been breached for higher degrees. The purpose of this paper is to obtain a sub-Weyl bound using the $GL(2)$ delta method, as introduced in \cite{Mu5}, \cite{Mu16}. In fact we will add one more layer in this method by introducing an extra averaging over the spectrum. This is a conductor lowering mechanism and it is effective to deal with the subconvexity problem in the $t$-aspect (and spectral aspect). We now state the main result of this paper. (We have not tried to obtain the best possible exponent.)\\

\begin{theorem}
\label{mthm}
Let $t>1$. Suppose $F$ is a Hecke cusp form for $SL(2,\mathbb{Z})$,
then 
\begin{align}
\label{main-bound}
L\left(\tfrac{1}{2}+it,F\right)\ll t^{\frac{1}{3}-\frac{1}{1200}+\varepsilon}.
\end{align}
\end{theorem}

 \bigskip
 
\noindent  
This is the first instance of a sub-Weyl bound for an $L$-function beyond degree one. The reader will notice that our argument also works for Eisenstein series (and possibly can be extended to Maass forms) and in particular yields a weak, nevertheless sub-Weyl, bound for the Riemann zeta function with exponent $1/6-1/2400$. But refining our counting method and complementing our argument by the classical theory of exponent pairs (not even going beyond Titchmarsh \cite{T}), we can obtain far better bounds. Further results in this direction will appear in an upcoming paper. Also it is conceivable that our new method can be used to break the long standing Voronoi barrier $O(x^{1/3+\varepsilon})$ for the `divisor problem' for cusp forms 
$$\sum_{n\leq x}\lambda_F(n).$$

\bigskip


\section{The set-up}

Suppose $F$ is a holomorphic Hecke cusp form of weight $k_0$ for $SL(2,\mathbb{Z})$ with normalised Fourier coefficients $\lambda_F(n)$, so that the Fourier expansion is given by
$$
F(z)=\sum_{n=1}^\infty \lambda_F(n)n^{(k_0-1)/2}e(nz)
$$
with $e(z)=e^{2\pi iz}$. The associated Hecke $L$-function is given by the Dirichlet series
$$
L(s,F)=\sum_{n=1}^\infty \frac{\lambda_F(n)}{n^s}
$$
in the half plane $\text{Re}(s)=\sigma>1$. This extends to an entire function and satisfies the Hecke functional equation. A consequence of which is the following bound.\\

\begin{Lemma}
We have
\begin{align*}
L\left(\tfrac{1}{2}+it,F\right)\ll t^\varepsilon \sup_N \frac{|S(N)|}{N^{1/2}} + t^{(1-\theta)/2}
\end{align*}
where the supremum is taken over $t^{1-\theta}<N<t^{1+\varepsilon}$, and $S(N)$ are sums of the form
\begin{align*}
S(N)=\sum_{N<n\leq 2N} \lambda_F(n)n^{it}.
\end{align*}
\end{Lemma}

\begin{proof}
This follows from the approximate functional equation, which itself is consequence of the functional equation.
\end{proof}

\begin{remark}
In fact one has a smoothed version of the sum $S(N)$, but we will not require that extra advantage.
\end{remark}

 The trivial bound for $S(N)$ gives the convexity bound $L(1/2+it,F)\ll |t|^{1/2+\varepsilon}$ for $|t|>1$. The generalised Riemann Hypothesis predicts square-root cancellation in these sums. For sub-Weyl one would need to show strong cancellations in $S(N)$. To this end, our first step consists of introducing the Weyl shifts
\begin{align*}
S(N)=\sum_{N<n\leq 2N} \lambda_F(n+h)(n+h)^{it}+O(Ht^\varepsilon),
\end{align*}
where $h\sim H\ll \sqrt{N}t^{1/3-\delta}$ for some $\delta>0$. (For the error term we are applying the Deligne bound, but all one needs is a Ramanujan bound on average.) It follows that 
$$
S(N)=S^\star(N)+O(Ht^\varepsilon)
$$
where
$$
S^\star(N)=\frac{1}{H}\sum_{h\in\mathbb{Z}}W\left(\frac{h}{H}\right)\:\sum_{N<n\leq 2N} \lambda_F(n+h)(n+h)^{it},
$$
with $W$ a non-negative valued smooth bump function (i.e. $W^{(j)}\ll_j 1$) supported on $[1,2]$ with $\int W=1$.\\

We will use the $GL(2)$ delta method  to analyse the sum $S^\star(N)$. Let $\mathcal{Q}\subset [Q,2Q]$ be a chosen set of prime numbers with $|\mathcal{Q}|\gg Q^{1-\varepsilon}$. For $q\in \mathcal{Q}$ let $\psi$ be an odd character of $\mathbb{F}_q^\times$. Let $H_k(q,\psi)$ be the set of Hecke-normalized newforms which is an orthogonal Hecke basis of the space of cusp forms $S_k(q,\psi)$. We will use the Petersson trace formula. We recall the following standard notations - $\lambda_f(n)$ denotes the Hecke-normalised $n$-th Fourier coefficient of the form $f$, 
$\omega_f^{-1}$ denotes the spectral weight, $S_\psi(a,b;c)$ denotes the generalised Kloosterman sum and $J_{k-1}(x)$ denotes the Bessel function of order $k-1$. 
Consider the (Fourier) sum 
\begin{align}
\label{F}
\mathcal{F} =&\sum_{\substack{k=1\\k\:\text{odd}}}^\infty W\left(\frac{k-1}{K}\right) \;\sum_{q\in \mathcal{Q}}\;\sum_{\substack{\psi\bmod{q}\\\psi(-1)=- 1}}\:\sum_{f\in H_k(q,\psi)}\omega_f^{-1}\\
\nonumber\times &\mathop{\sum\sum}_{m,\ell=1}^\infty \lambda_F(m)\lambda_f(m)\psi(\ell)U\left(\frac{m\ell^2}{N}\right)\\
\nonumber \times &\mathop{\sum\sum}_{\substack{N<n\leq 2N\\ h\in\mathbb Z}} \overline{\lambda_f(n
+h)}(n+h)^{it}\:W\left(\frac{h}{H}\right).
\end{align}
Here  $U$ is a smooth function supported in $[1/2,3]$, with $U(x)=1$ for $x\in [1,2]$, and satisfying $U^{(j)}(y)\ll_j 1$. The (direct) off-diagonal is defined as
\begin{align}
\label{O}
\mathcal{O} =&\sum_{\substack{k=1\\k\:\text{odd}}}^\infty W\left(\frac{k-1}{K}\right) \;\sum_{q\in\mathcal{Q}}\;\sum_{\substack{\psi\bmod{q}\\\psi(-1)=- 1}}\:\\
\nonumber\times &\mathop{\sum\sum}_{m,\ell=1}^\infty \lambda_F(m)\psi(\ell)U\left(\frac{m\ell^2}{N}\right)\:\mathop{\sum\sum}_{\substack{N<n\leq 2N\\ h\in\mathbb Z}} (n+h)^{it}W\left(\frac{h}{H}\right)\\
\nonumber \times & \sum_{c=1}^\infty \frac{i^{-k}S_\psi(n+h,m;cq)}{cq}J_{k-1}\left(\frac{4\pi\sqrt{m(n+h)}}{cq}\right).
\end{align}\\

\begin{Lemma}
\label{lemma:f-o}
We have
\begin{align}
\label{s-star-n}
\frac{|S(N)|}{\sqrt{N}}\ll \frac{|\mathcal{F}|+|\mathcal{O}|}{\sqrt{N}HQ^2K}+\frac{Ht^\varepsilon}{N^{1/2}}.
\end{align}
\end{Lemma}

\begin{proof}
We apply the Petersson trace formula to $\mathcal{F}$. The diagonal term  corresponds to $m=n+h$, in which case we are left with the weight function $U((n+h)\ell^2/N)$, with $N+H<n+h\leq 2N+2H$. Considering the supports of the functions, we see that the weight function is vanishing if $\ell> 1$. Hence the diagonal is given by
\begin{align*}
&\sum_{\substack{k=1\\k\:\text{odd}}}^\infty W\left(\frac{k-1}{K}\right)\sum_{q\in \mathcal{Q}}\:\sum_{\substack{\psi\bmod{q}\\\psi(-1)=- 1}}\; \mathop{\sum\sum}_{\substack{N<n\leq 2N\\ h\in \mathbb Z}} \lambda_F(n+h)(n+h)^{it}W\left(\frac{h}{H}\right)
\end{align*}
which equals $F^\star S^\star(N)$,
where
\begin{align*}
F^\star =H\sum_{\substack{k=1\\(-1)^k=-1}}^\infty W\left(\frac{k-1}{K}\right)\;\sum_{q\in\mathcal{Q}}\sum_{\substack{\psi\bmod{q}\\\psi(-1)=- 1}}\;1\asymp HQ^2K.
\end{align*}
The off-diagonal in the Petersson formula yields the sum $\mathcal{O}$. The lemma follows.
\end{proof}

In the rest of the paper we will prove sufficient bounds for the off-diagonal $\mathcal{O}$ and the Fourier sum $\mathcal{F}$.

\bigskip

\section{The off-diagonal}
\label{off-diag-sec}
In this section we will analyse the off-diagonal $\mathcal{O}$ which is given in \eqref{O}. Our aim is to prove Proposition~\ref{prop-1} which is stated at the end of this section. We begin by recalling the following formula for the sum of Bessel functions.\\

\begin{Lemma}
\label{lemma:sum-bessel}
Suppose $g$ is a compactly supported smooth function on $\mathbb{R}$ and $x\in \mathbb{R}$. We have
$$
4\sum_{u\equiv a\bmod{4}} g(u)J_u(x)=\int_\mathbb{R}\hat{g}(v)c_a(v;x)\mathrm{d}v
$$
where
$$
c_a(v;x)=-2i\sin(x\sin 2\pi v)+2i^{1-a}\sin(x\cos 2\pi v)
$$
and the Fourier transform is defined by
$$
\hat{g}(v)=\int g(u)e(uv)\mathrm{d}u.
$$
\end{Lemma}

\begin{proof}
See \cite{Iw-book}.
\end{proof}

The following lemma from \cite{BKY}, which provides a criterion for an exponential integral to be negligibly small, will also play a crucial role in our analysis.\\

\begin{Lemma}
\label{lemma:small-int}
Let $Y\geq 1$, $X,Q,U,R>0$. Let $w$ be a smooth weight function supported on $[\alpha,\beta]$, satisfying $w^{(j)}\ll_j XU^{-j}$, and let $h$ be a real valued smooth function on the same interval such that $|h'|\geq R$, and $h^{(j)}\ll_j YQ^{-j}$ for $j\geq 2$. Then 
\begin{align}
\label{lem-neg}
\int_{\mathbb{R}} w(y)e(h(y))\mathrm{d}y \ll_A (\beta-\alpha)X\left[\left(\frac{QR}{\sqrt{Y}}\right)^{-A}+(RU)^{-A}\right],
\end{align}
for any $A\geq 1$. In particular the integral is `negligibly small', i.e. $O(t^{-A})$ for any $A>1$, if $R\gg t^\varepsilon\max\{Y^{1/2}/Q, 1/U\}$ and $t\gg [(\beta-\alpha)X]^\varepsilon$.
\end{Lemma}

For $C\ll Nt^\varepsilon/QK^2$ dyadic we introduce the family of sums
\begin{align}
\label{O-after-k-sum}
\mathcal{O}(C):=&\frac{K}{\sqrt{NCQ}}\sum_{q\in\mathcal{Q}}\;\sum_{\substack{\psi\bmod{q}\\\psi(-1)=- 1}}\:\mathop{\sum\sum}_{m,\ell=1}^\infty \lambda_F(m)\psi(\ell)U\left(\frac{m\ell^2}{N}\right)\\
\nonumber \times &\mathop{\sum\sum}_{\substack{N<n\leq 2N\\ h\in\mathbb Z}} (n+h)^{it}W\left(\frac{h}{H}\right)\: \sum_{c\sim C} S_\psi(n+h,m;cq)e\left(\frac{2\sqrt{m(n+h)}}{cq}\right).
\end{align} 

\begin{remark}[Notation] Let us recall the following notation. Suppose $S$ is a sum we seek to estimate, and suppose we have a family of sums $\mathfrak{F}=\{T_f\}$ such that $|\mathfrak{F}|\ll t^\varepsilon$, and
$$
S\ll t^\varepsilon \sum_{\mathfrak{F}} |T_f| +t^{-A}
$$
for any $A>0$. Then we simply write $S\triangleleft T_f$.
\end{remark}

\begin{Lemma}
\label{lemma:od}
Suppose $Q$, $K$ are such that
\begin{align}
\label{1st-cond-QK}
N\ll QK^4t^{-\varepsilon}.
\end{align}  
Then (in the above notation) we have $\mathcal{O}\triangleleft \mathcal{O}(C)$.
\end{Lemma}

\begin{proof}
Consider the sum over $k$ in \eqref{O}
\begin{align}
\label{k-sum-in}
	\sum_{\substack{k=1\\k\:\text{odd}}}^\infty W\left(\frac{k-1}{K}\right)i^{-k} J_{k-1}\left(\frac{4\pi\sqrt{m(n+h)}}{cq}\right).
\end{align}
Applying Lemma~\ref{lemma:sum-bessel} we get that this sum is given by
\begin{align*}
	\int_{\mathbb R}\hat{W}(v)\sin\left(2\pi x\cos \frac{2\pi v}{K}\right)\mathrm{d}v,
\end{align*}
with $x=2\sqrt{m(n+h)}/cq$. This integral can be expressed as a linear combination of the integrals
\begin{align*}
	\int_{\mathbb R}\hat{W}(v)e\left(\pm x\cos \frac{2\pi v}{K}\right)\mathrm{d}v.
\end{align*}
Since the Fourier transform $\hat{W}$ decays rapidly, we can and will introduce a smooth function $F$ with support $[-V,V]$ with $F^{(j)}\ll_j V^{-j}$, $F(v)=1$ for $|v|\ll t^\varepsilon$ and any $V$ in the range $t^\varepsilon\ll V\ll Kt^{-\varepsilon}$, and replace the above integrals (up to negligible error terms) by
\begin{align*}
	\iint_{\mathbb R^2}W(u)F(v)e\left(uv\pm x\cos \frac{2\pi v}{K}\right)\mathrm{d}u\mathrm{d}v.
\end{align*}
Now we apply Lemma~\ref{lemma:small-int} to the $v$-integral, and it follows that the integral is negligibly small, i.e. $O_J(t^{-J})$ for any $J\geq 1$, if $x\ll K^{2-\varepsilon}$. This analysis holds even if the weight function $W$ has a little oscillation, say $W^{(j)}\ll_j t^{j\varepsilon}$. In the complementary range for $x$ we expand the cosine function into a Taylor series. Since $x\ll N/Q$, if we assume \eqref{1st-cond-QK} then we only need to retain the first two terms in the expansion, and the above integral essentially reduces to
\begin{align*}
	e(\pm x)\iint_{\mathbb R^2}W(u)F(v)e\left(uv\mp \frac{4\pi^2 xv^2}{K^2}\right)\mathrm{d}u\mathrm{d}v.
\end{align*}
To the integral over $v$ we apply the stationary phase analysis. It turns out to be negligibly small (due to \eqref{lem-neg}) when we have $+$ sign inside the exponential, otherwise the integral essentially reduces to 
$$
e\left(x+\frac{u^2K^2}{16\pi^2x}\right)\frac{K}{\sqrt{x}}\leadsto e(x)\frac{K}{\sqrt{x}}
$$
with $x\gg K^{2-\varepsilon}$ (upto an oscillatory factor which `oscillates at most like $t^\varepsilon$'). In any case, it follows that we can cut the sum over $c$ in \eqref{O} at $C\ll Nt^\varepsilon/QK^2$, at a cost of a negligible error term.  The lemma follows.
\end{proof}

\begin{remark}
To get the Weyl bound it is sufficient to take $QK^2\gg Nt^\varepsilon$, so that the off-diagonal is trivially small (see \cite{AKMS}). But to achieve sub-Weyl one needs to take $QK^2$ smaller. 
\end{remark}

\begin{Lemma}
\label{lemma:oc-after-poisson}
Suppose $H=N/t^{1/3}$. We have
$$
\mathcal{O}(C)\ll \sum_\pm |\mathcal{O}_\pm(C)^\dagger|,
$$
where
\begin{align}
\label{O-after-poisson}
&\mathcal{O}_\pm(C)^\dagger:=\frac{KH\sqrt{Q}}{\sqrt{NC}t^{1/6}}\sum_{q\in\mathcal{Q}}\:\mathop{\sum\sum}_{m,\ell=1}^\infty \lambda_F(m)U\left(\frac{m\ell^2}{N}\right)\\
\nonumber &\times \sum_{c\sim C}\mathop{\sum\sum}_{\substack{N<n\leq 2N\\ h\equiv \pm \bar{\ell}\bmod{q}}} \:n^{it}e\left(\frac{2\sqrt{mn}}{cq}-\frac{hn}{cq}-\frac{\bar{h}m}{cq}+\frac{A^2}{B}\right)\;W\left(\frac{2nA}{HB}\right),
\end{align}
with
\begin{align*}
A=\frac{t}{2\pi}+\frac{\sqrt{mn}}{cq}-\frac{hn}{cq},\;\;\;\text{and}\;\;\;B=\frac{t}{\pi}+ \frac{\sqrt{mn}}{cq}.
\end{align*}
\end{Lemma}

\begin{proof}
We apply the Poisson summation formula on the sum over the shifts $h$ in $\mathcal{O}(C)$. This transforms the sum
\begin{align}
\label{h-sum-direct-od}
\mathop{\sum}_{\substack{h\in\mathbb Z}} S_\psi(n+h,m;cq)e\left(\frac{2\sqrt{m(n+h)}}{cq}\right)(n+h)^{it}W\left(\frac{h}{H}\right)
\end{align} 
into
\begin{align}
H\sum_{\substack{h\in\mathbb{Z}\\(h,cq)=1}} \psi(-h)\:e\left(-\frac{hn}{cq}-\frac{\bar{h}m}{cq}\right) \mathfrak{I}
\end{align}
where the integral is given by
\begin{align*}
\mathfrak{I}=\int W(y) e(f(y))\mathrm{d}y,
\end{align*}
with
\begin{align*}
f(y)= \frac{t}{2\pi}\log (n+Hy)+\frac{2\sqrt{m(n+Hy)}}{cq}-\frac{Hhy}{cq}.
\end{align*}
We have the Taylor expansion
\begin{align}
\label{f-dual-initial}
f(y)= &\left[\frac{t}{2\pi}\log n+\frac{2\sqrt{mn}}{cq}\right]+\frac{Hy}{n}\left[\frac{t}{2\pi}+\frac{\sqrt{mn}}{cq}-\frac{hn}{cq}\right]\\
\nonumber &-\frac{H^2y^2}{n^2}\left[\frac{t}{4\pi}+ \frac{\sqrt{mn}}{4cq}\right]+O\left(\frac{t}{(N/H)^3}\right).
\end{align}
To restrict the phase function up to the quadratic term we pick $H=N/t^{1/3}$. (The tail term $E(\dots)$ is a `flat function' in the sense that $y^jE^{(j)}(y)\ll 1$ with respect to all the variables.)  Then by the stationary phase analysis we can replace $\mathcal{O}(C)$ by
\begin{align*}
&\frac{KH}{\sqrt{NCQ}t^{1/6}}\sum_{q\in\mathcal{Q}}\;\sum_{\substack{\psi\bmod{q}\\\psi(-1)=- 1}}\:\mathop{\sum\sum}_{m,\ell=1}^\infty \lambda_F(m)\psi(\ell)U\left(\frac{m\ell^2}{N}\right)\\
\nonumber &\times \sum_{c\sim C}\:\mathop{\sum\sum}_{\substack{N<n\leq 2N\\ h\in\mathbb Z}} \psi(h)\:n^{it}e\left(\frac{2\sqrt{mn}}{cq}-\frac{hn}{cq}-\frac{\bar{h}m}{cq}+\frac{A^2}{B}\right)\;W\left(\frac{2nA}{HB}\right).
\end{align*} 
Finally executing the sum over $\psi$ we arrive at the expression given in \eqref{O-after-poisson}. The lemma follows.
\end{proof}

\begin{remark}
It should be possible to improve our result by picking $H=N/t^{1/4}$ and allowing up to the cubic term. The argument given below might go through with some modifications. But the expressions would be much more complicated.
\end{remark}

\begin{remark}
Notice that the weight function in \eqref{O-after-poisson} puts the restriction that $A\asymp t^{2/3}$, and consequently 
\begin{align}
\label{n-rest-direct-od}
\left|n-\frac{tcq}{2\pi h}\right|\ll N\max \left\{\frac{1}{t^{1/3}},\frac{N}{C\ell Qt}\right\}=:N\Delta_0.
\end{align}
and $h\sim tCQ/N$. 
\end{remark}

Below we continue our analysis with $\mathcal{O}_+(C)^\dagger$ which we denote by $\mathcal{O}(C)^\dagger$. The analysis of the other term $\mathcal{O}_-(C)^\dagger$ is exactly same.\\

\begin{Lemma}
We have
\begin{align}
\label{direct-od-cauchy}
\mathcal{O}(C)^\dagger\ll \sum_\ell\;\frac{KH\sqrt{Q}}{\sqrt{C}t^{1/6}\ell}\;\Omega_\ell^{1/2}
\end{align}
where $\Omega_\ell$ is given by
\begin{align}
\label{off-diag-omega}
\mathop{\sum}_{m\sim N/\ell^2}\:\Bigl| \sum_{q\in\mathcal{Q}} \sum_{c\sim C}\mathop{\sum\sum}_{\substack{N<n\leq 2N\\ h\equiv \bar{\ell}\bmod{q}}} \:n^{it}e\left(\frac{2\sqrt{mn}}{cq}-\frac{hn}{cq}-\frac{\bar{h}m}{cq}+\frac{A^2}{B}\right)\;W\left(\frac{2nA}{HB}\right)\Bigr|^2.
\end{align} 
\end{Lemma}

\begin{proof}
The lemma follows by applying the Cauchy inequality and the Ramanujan bound on average for the Fourier coefficients.
\end{proof}

Next we open the absolute value square (after smoothing the outer sum) leading to an expression of the form
$$
\mathop{\sum}_{m\in\mathbb{Z}}\:W(m\ell^2/N)\:\mathop{\sum\sum}_{q_1,q_2\in\mathcal{Q}} \mathop{\sum\sum}_{c_1,c_2\sim C}\mathop{\sum\sum}_{\substack{N<n_1\leq 2N\\ h_1\equiv \bar{\ell}\bmod{q_1}}}\mathop{\sum\sum}_{\substack{N<n_2\leq 2N\\ h_2\equiv \bar{\ell}\bmod{q_2}}}\;\{\dots\}.
$$ 
Our argument now depends on the size of $c_1q_1h_1-c_2q_2h_2$. In the case of small gap, i.e.
\begin{align}
\label{gap-small}
c_1q_1h_1-c_2q_2h_2\ll \Delta_0\frac{t(CQ)^2}{N}=:M/\ell,
\end{align}
we estimate the sum trivially. The problem reduces to a non-trivial counting problem. However as we are not trying to get the best possible exponent here, we will settle with a rather easy bound. For the case of large gap, i.e. in the range complementary to \eqref{gap-small}, we apply the Poisson summation formula on the sum over $m$. \\

\begin{Lemma}
\label{lemma:od-small}
Suppose $Q, K\ll t^{1/3}$. The total contribution to \eqref{direct-od-cauchy} of the terms satisfying \eqref{gap-small} is bounded by
$$
\sqrt{N}KHQ^2t^{1/3}\;\frac{N^{3/2}t^{1/6}}{Q^{5/2}K^3}.
$$
\end{Lemma}

\begin{proof}
The number of $(n_1,n_2)$ pairs is given by $1+(\Delta_0 N)^2\asymp (\Delta_0 N)^2$ (as $N\gg t^{1/3}$) once the other variables are given. Now we need to count the number of $(q_i,c_i,h_i)$. Setting $h_i\ell=1+g_iq_i$ we get that 
\begin{align}
\label{inequality-to-count}
c_1g_1q_1^2-c_2g_2q_2^2\ll M.
\end{align}
We set $a_i=c_ig_i\sim tC^2/N$. Fix $a_1$, $a_2$, and set $\alpha=\sqrt{a_2/a_1}$. We seek to count the number of $q_i$ such that $|q_1-\alpha q_2|\ll MN/tC^2Q$ and $q_1g_1\equiv -1\bmod{\ell}$.  Given $q_2$ there are at most $O(1+MN/tC^2Q\ell)$ many $q_1$'s, and so 
$$
\#\{(q_1,q_2,c_1,c_2,h_1,h_2)\}\ll \left(\frac{tC^2}{N}\right)^2 \left(1+\frac{MN}{tC^2Q\ell}\right)Q.
$$
It follows that the contribution of these terms to \eqref{direct-od-cauchy} is bounded by
$$
\sum_\ell\;\frac{KH\sqrt{Q}}{\sqrt{C}t^{1/6}\ell}\:\Delta_0 N \frac{tC^2}{N} \left(1+\sqrt{\frac{MN}{tC^2Q\ell}}\right)\sqrt{Q}.
$$
Next using the fact that $\Delta_0 C\ell\ll N/QK^2 t^{1/3}$ (as $K\ll t^{1/3}$) we get that the expression is dominated by
$$
\sqrt{N}KHQ^2t^{1/3}\sum_\ell \frac{(1+\sqrt{\Delta_0 Q})\Delta_0 t^{1/2}C^{3/2}}{Q\ell}\ll \sqrt{N}KHQ^2t^{1/3}\left(\frac{t^{1/6}N^{3/2}}{Q^{5/2}K^3}+\frac{N^{3/2}}{Q^2K^3}\right).
$$
The lemma follows as the first term dominates the other since $Q\ll t^{1/3}$.
\end{proof}

\begin{remark}
The bound in Lemma~\ref{lemma:od-small} is satisfactory if
\begin{align}
\label{range-N-1}
N\ll Q^{5/3}K^{2}t^{-1/9-2\delta/3}.
\end{align}
This prompts us to pick 
\begin{align}
\label{pick-q}
Q=N^{3/5}t^{1/15+2\delta/5}K^{-6/5}.
\end{align}
This is valid if we have $1\leq Q\ll t^{1/3}$. The lower bound holds if $N\gg K^2t^{-1/9-2\delta/3}$ which is fine as later (see \eqref{pick-k}) we will choose $K=\min\{t^{1/3-2\delta},Nt^{-1/3-2\delta}\}$. For the upper bound we require to take $\delta<1/42$.
\end{remark}

\begin{remark}
It should be possible to improve the above lemma. Indeed, in generic case $MN/tC^2Q$ is much smaller than one. So $q_1$ exists only if  $q_2$ is such that $\|\alpha q_2\|\ll MN/tC^2Q$. So, roughly speaking, the count for the number of solutions of \eqref{inequality-to-count} is given by
\begin{align*}
\mathop{\sum\sum}_{a_1,a_2}\;\sum_{q_2}\;\mathbf{1}_{\|\alpha q_2\|\ll MN/tC^2Q}.
\end{align*}
In the above lemma we have estimated this sum trivially. But one may use exponential sum to detect the condition, and try to get cancellation in that sum on average.
\end{remark}

\begin{Lemma}
\label{lemma:od-big}
The contribution to \eqref{direct-od-cauchy} of the terms not satisfying \eqref{gap-small} is bounded by
$$
\sqrt{N}KHQ^2t^{1/3}\:\frac{Nt^{1/6}}{K^{5/2}Q^{3/2}}.
$$
\end{Lemma}

\begin{proof}
In the complementary range when the gap is not small, i.e. \eqref{gap-small} does not hold, we apply the Poisson summation on the sum over $m$. We get that the contributions of these terms to \eqref{off-diag-omega} is bounded by
\begin{align}
\label{omega-after-poisson-od}
 \frac{N}{\ell^2}\;&\mathop{\mathop{\sum}_{m\in\mathbb{Z}}\mathop{\sum\sum}_{q_1,q_2\in  \mathcal{Q}}\mathop{\sum\sum}_{c_1,c_2\sim C}\mathop{\sum\sum}_{N<n_1,n_2\leq 2N}\mathop{\sum\sum}_{h_i\equiv  \bar{\ell}\bmod{q_i}}}_{\substack{\overline{h_1}\:c_2q_2- \overline{h_2}\:c_1q_1\equiv m\bmod{c_1c_2q_1q_2}\\ c_1q_1h_1-c_2q_2h_2\gg M/\ell}} \:|\mathfrak{H}|
\end{align}
where the integral is given by
\begin{align*}
&\mathfrak{H}=\int W(y)\:W\left(\frac{2n_1A_1}{HB_1}\right)W\left(\frac{2n_2A_2}{HB_2}\right)\:e\Bigl[A_1^\star y+ A_2^\star y^2+A_3^\star y^3+\dots\Bigr]\mathrm{d}y.
\end{align*}
The coefficients in the Taylor expansion can be computed explicitly, and we get
$$
A_1^\star =\frac{2\sqrt{N}}{\ell}\left(\frac{\sqrt{n_1}}{c_1q_1}-\frac{\sqrt{n_2}}{c_2q_2}\right)+O\left(E\right)
$$
and
$$
A_2^\star =\frac{N}{\ell^2}\left(\frac{\pi}{t}\left(\frac{n_1}{(c_1q_1)^2}-\frac{n_2}{(c_2q_2)^2}\right)-\frac{m}{c_1c_2q_1q_2}\right)+O(E),
$$
and $A_j^\star=O(E)$ for all $j\geq 3$, 
where
$$
E=\frac{N\Delta_0}{QC\ell}.
$$
Note that the $j$-th derivative of the weight function in the integral is bounded by $(N/t^{1/3}QC\ell)^j\ll E^j$. Since we are in the case where \eqref{gap-small} does not hold, we get that the first term of $A_1^\star$ is larger than the error term $E$ (recall \eqref{n-rest-direct-od}), and hence it follows using Lemma~\ref{lemma:small-int} that the integral is negligibly small unless 
$$
0\neq m\asymp \frac{(QC\ell)^2}{N}\:\frac{\sqrt{N}}{\ell}\left(\frac{\sqrt{n_1}}{c_1q_1}-\frac{\sqrt{n_2}}{c_2q_2}\right)=M_s.
$$ 
In this case we estimate the integral using the second derivative bound 
$$
\mathfrak{H}\ll \frac{C\ell Q}{\sqrt{N|m|}}.
$$
It now remains to count the number of vectors $(m,h_1,h_2,c_1,c_2,q_1,q_2)$ satisfying the congruence conditions. This is not an easy task, and we only seek to obtain a good upper bound. Given $(m,c_1,q_1)$, there are at most $O(tC/N)$ many $h_1$. Then $c_2q_2$ is determined modulo $c_1q_1$. Hence there are at most $O(t^\varepsilon)$ many $(c_2,q_2)$. Finally we get at most $O(t/N)$ many $h_2$. It follows that the number of vectors is bounded by 
$$
t^\varepsilon\: M_s CQ\:\left(\frac{tC}{N}\right)\:\left(\frac{t}{N}\right)\ll t^\varepsilon\: M_s \:Q\:\left(\frac{tC}{N}\right)^2.
$$
Then we count the number of $n_i$ using the restriction \eqref{n-rest-direct-od}. It follows that the contribution of these terms without `small gap', to \eqref{direct-od-cauchy} is dominated by
$$
\sum_\ell \frac{KH\sqrt{Q}}{\sqrt{C}t^{1/6}\ell}\;\frac{N^{1/4}C^{7/4}Q^{5/4}t^{2/3}}{\ell^{1/4}}.
$$
This is dominated by the bound given in the statement of the lemma.
\end{proof}

\begin{remark}
The bound given in Lemma~\ref{lemma:od-big} is satisfactory if
\begin{align}
\label{range-N-2}
N\ll Q^{3/2}K^{5/2}t^{-1/6-\delta}.
\end{align} 
If we pick $Q$ according to \eqref{pick-q}, and $K$ according to \eqref{pick-k} then the above inequality holds if $\delta<1/30$. 
Also note that the bound in the lemma is off from the expected bound by a factor of $Q^{1/2}$, as we lost a congruence modulo $q_2$ in our count.
\end{remark}

We now state in form of a proposition, what we have achieved in this section\\

\begin{proposition}
\label{prop-1}
Let $H=N/t^{1/3}\ll N^{1/2}t^{1/3-\delta}$. Let $Q$ be as given in \eqref{pick-q} and $K$ be as given in \eqref{pick-k}. 
We have 
$$\mathcal{O}\ll \sqrt{N}HQ^2K t^{1/3-\delta}
$$
if $\delta<1/42$. 
\end{proposition} 

\bigskip


\section{Applying functional equation: Dual side}
\label{fe}

In the rest of the paper we will seek to prove sufficient bound for $\mathcal{F}$. Our aim is to prove Proposition~\ref{prop-2} which we state at end of the paper. \\

Let us introduce the family of dual sums
\begin{align}
\label{to-ana}
&\mathcal{D} =\frac{N}{K^2}\sum_{\substack{k=1\\k\:\text{odd}}}^\infty W\left(\frac{k-1}{K}\right) \;\sum_{n\sim N}\sum_{\substack{\psi\bmod{q}\\\psi(-1)=- 1}}\varepsilon_\psi^2\;\sum_{f\in H_k(q,\psi)}\omega_f^{-1}\\
\nonumber\times &\mathop{\sum\sum}_{m\ell^2\sim \tilde{N}}\: \lambda_F(m)\overline{\lambda_f(mq^2)}\:\bar{\psi}(\ell)\;\mathop{\sum}_{h\in\mathbb Z} \overline{\lambda_f(n+h)}(n+h)^{it}W\left(\frac{h}{H}\right),
\end{align}
with $\tilde{N}\asymp Q^2K^4/N$. (In this paper the notation $A\asymp B$ means that $B/t^\varepsilon\ll A\ll Bt^\varepsilon$, with implied constants depending on $\varepsilon$.) \\

\begin{Lemma}
\label{lemma:dualize}
We have
$$
\mathcal{F}\triangleleft \mathcal{D}.
$$
\end{Lemma}

\begin{proof}
Consider the Fourier sum $\mathcal{F}$ (as given in\eqref{F}), where we will apply the functional equation of the $L$-function $L(s,F\times f)$, to dualise the sum over $(m,\ell)$. By the Mellin inversion formula we get
\begin{align}
\label{sum-1}
&\mathop{\sum\sum}_{m,\ell=1}^\infty \lambda_F(m)\lambda_f(m)\psi(\ell)U\left(\frac{m\ell^2}{N}\right)=\frac{1}{2\pi i}\int_{(2)}N^s\tilde{U}(s)L(s,F\times f)\mathrm{d}s,
\end{align}
where $\tilde{U}$ stands for the Mellin transform of $U$. Using the functional equation of the $L$-function $L(s,F\times f)$ we get
\begin{align*}
\frac{1}{2\pi i}\int_{(2)}N^s\tilde{U}(s)\;\eta^2 \left(\frac{q}{4\pi^2}\right)^{1-2s}\frac{\gamma_k(1-s)}{\gamma_k(s)} L(1-s,F\times \bar f)\mathrm{d}s,
\end{align*}
where 
$$\gamma_k(s)=\Gamma\left(s+\frac{(k-k_0)}{2}\right)\Gamma\left(s+\frac{(k+k_0)}{2}-1\right).$$
The sign of the functional equation is given by 
\begin{align*}
\eta^2=i^{2k}\frac{g_\psi^2}{\lambda_f(q^2)q}
\end{align*}
where $g_\psi$ is the Gauss sum associated with the character $\psi$. Next we expand the $L$-function into a Dirichlet series and take dyadic subdivision. By shifting contours to the right or left we can show that the contribution of the terms from the blocks with $m\ell^2 \notin [Q^2K^4t^{-\varepsilon}/N,Q^2K^4t^{\varepsilon}/N]$ is negligibly small.  Hence the sum in \eqref{sum-1} essentially gets transformed into
\begin{align}
\label{sum-11}
\varepsilon_\psi^2\:&\frac{N}{QK^2}\:\frac{\gamma_k(1/2+i\tau)}{\gamma_k(1/2-i\tau)}\: \mathop{\sum\sum}_{m,\ell=1}^\infty \lambda_F(m)\overline{\lambda_f(mq^2)}\bar{\psi}(\ell)U\left(\frac{m\ell^2}{\tilde N}\right)
\end{align}
where $\varepsilon_\psi$ is the sign of the Gauss sum $g_\psi$ and $\tilde N\asymp \frac{Q^2K^4}{N}.
$
We are keeping a $\tau$, with $|\tau|\ll t^\varepsilon$, in the gamma factor as we need to keep track of possible oscillation in the $k$ aspect.
Now let us study the gamma factor. Using Stirling series 
$$
\Gamma(z)=\sqrt{\frac{2\pi}{z}}\left(\frac{z}{e}\right)^z\left(\sum_{j=1}^J\frac{a_j}{z^j}+O(|z|^{-J})\right)
$$
which holds for $z=k/2+i\tau$ as above, 
it turns out that this ratio of the gamma functions is essentially equivalent to  
$$
e(\arg [(k/2+i\tau)^{k/2+i\tau}]/\pi)=e(\tau \log (k^2/4+\tau^2)/2\pi +k\arg (k/2+i\tau)/2\pi).
$$
Now expanding $\log(k^2/4+\tau^2)=2\log (k/2)+O(\tau^2/k^2)$, and 
$$\arg(k/2+i\tau)=\tan^{-1}(2\tau/k)=2\tau/k+O(\tau^3/k^3),$$ 
we get that the ratio of the gamma functions essentially behaves like  $k^{4i\tau}$. So the gamma factor oscillates mildly, which can be neglected. The lemma follows.
\end{proof}

\begin{remark}
Observe that we have dropped some of the smooth weight functions, as they do not play any role whatsoever in the upcoming analysis. Also we have dropped the $q$ sum, at the cost of a multiplier of size $Q$. Indeed the $q$ sum is not involved in the dual side at all. 
\end{remark}

\begin{Lemma}
Suppose 
\begin{align}
\label{2nd-cond-QK}
Q\ll K^2t^{-\varepsilon}.
\end{align}
We have
$$
\mathcal{D}\triangleleft \mathcal{O}^\star(C)
$$
where for $C\ll Qt^\varepsilon$ dyadic 
\begin{align}
\label{dual-od}
&\mathcal{O}^\star(C)= \frac{N\sqrt{\ell}}{\sqrt{CQ}K^2}\sum_{n\sim N}\:\mathop{\sum}_{\substack{m\sim \tilde{N}/\ell^2}} \lambda_F(m)\\
\nonumber &\times \mathop{\sum}_{h\in\mathbb Z} (n+h)^{it}W\left(\frac{h}{H}\right)\:\sum_{\substack{c\sim C\\(c,q)=1}} S(n+h,m;c)\:e\left(\frac{2\sqrt{m (n+h)}}{c}\right).
\end{align}
\end{Lemma}

\begin{proof}
In the expression for $\mathcal{D}$, one writes $m=q^\nu m'$ with $q\nmid m'$, so that $\overline{\lambda_f(m')}=\lambda_f(m')\bar{\psi}(m')$.
We get
\begin{align*}
&\frac{N}{K^2}\sum_{\nu=0}^\infty \lambda_F(q^\nu)\sum_{\substack{k=1\\k\:\text{odd}}}^\infty W\left(\frac{k-1}{K}\right) \; \sum_{n\sim N}\;\sum_{\substack{\psi\bmod{q}\\\psi(-1)=- 1}}\varepsilon_\psi^2\;\sum_{f\in H_k(q,\psi)}\omega_f^{-1}\\
\nonumber\times &\mathop{\sum\sum}_{\substack{m\ell^2\sim\tilde{N}}/q^\nu} \lambda_F(m)\lambda_f(m)\:\bar{\psi}(m\ell)\:\mathop{\sum}_{h\in\mathbb Z} \overline{\lambda_f(q^{2+\nu}(n+h))}(n+h)^{it}W\left(\frac{h}{H}\right).
\end{align*}
We apply the Petersson formula to this dual sum. There is no diagonal contribution as $\psi(m)=0$ when $q|m$. Hence we are only left with the (dual) off-diagonal which is given by 
\begin{align}
\label{O-star}
\mathcal{O}^\star =&\frac{N}{K^2}\sum_{\nu=0}^\infty \lambda_F(q^\nu)\sum_{\substack{k=1\\k\:\text{odd}}}^\infty W\left(\frac{k-1}{K}\right) \;\sum_{n\sim N}\;\sum_{\substack{\psi\bmod{q}\\\psi(-1)=- 1}}\:\varepsilon_\psi^2\\
\nonumber\times &\mathop{\sum\sum}_{\substack{m\ell^2\sim\tilde{N}/q^\nu}} \lambda_F(m)\bar{\psi}(m\ell)\:\mathop{\sum}_{h\in\mathbb Z} (n+h)^{it}V\left(\frac{h}{H}\right)\\
\nonumber \times &\sum_{c=1}^\infty \frac{i^{-k}S_\psi(q^{2+\nu}(n+h),m;cq)}{cq}J_{k-1}\left(\frac{4\pi\sqrt{mq^\nu(n+h)}}{c}\right).
\end{align}\\

Consider the sum over $k$ which is given by 
\begin{align}
\label{k-sum}
	\sum_{\substack{k=1\\k\:\text{odd}}}^\infty W\left(\frac{k-1}{K}\right)i^{-k} J_{k-1}\left(\frac{4\pi\sqrt{mq^\nu (n+h)}}{c}\right).
\end{align}
This sum is exactly same as we had for the direct off-diagonal before (only now it is independent of $q$ in the generic case $\nu=0$). Notice that we have deliberately dropped the gamma factors, which were mildly oscillating. So we assume that $W^{(j)}\ll_j t^{\varepsilon j}$ in the present case. Temporarily we set $x=2\sqrt{mq^\nu (n+h)}/c$. We choose to have \eqref{2nd-cond-QK} so that we can again restrict ourselves to the quadratic phase in the expansion, and the above sum essentially reduces to
\begin{align*}
	e(\pm x)\iint_{\mathbb R^2}W(u)F(v)e\left(uv\mp \frac{4\pi^2 xv^2}{K^2}\right)\mathrm{d}u\mathrm{d}v.
\end{align*}
As before this reduces to
$$
e(x)\frac{K}{\sqrt{x}}
$$
with $x\gg K^{2-\varepsilon}$. In the complementary range the integral is negligibly small. With this the off-diagonal essentially reduces to 
\begin{align}
& \frac{N\sqrt{\ell}}{\sqrt{C}K^2Q^{3/2}}\;\sum_{n\sim N}\;\sum_{\substack{\psi\bmod{q}\\\psi(-1)=- 1}}\:\varepsilon_\psi^2\:\sum_{\nu=0}^\infty \lambda_F(q^\nu)\mathop{\sum}_{m\sim \tilde{N}/q^\nu \ell^2}\lambda_F(m)\bar{\psi}(m\ell)\\
\nonumber \times \mathop{\sum}_{h\in\mathbb Z} &(n+h)^{it}W\left(\frac{h}{H}\right)\sum_{c\sim C} S_\psi(q^{2+\nu}(n+h),m;cq)\:e\left(\frac{2\sqrt{mq^\nu (n+h)}}{c}\right).
\end{align}
with $C\ll Qt^\varepsilon/\ell$. If $q\|c$ then the Kloosterman sum vanishes as $q\nmid m$, so we necessarily have $(c,q)=1$. The Kloosterman sum splits, and we get that
\begin{align*}
	\sum_{\substack{\psi\bmod{q}\\\psi(-1)=- 1}}\:\varepsilon_\psi^2\bar{\psi}(m\ell)S_\psi(q^{2+\nu}(n+h),m;cq)
\end{align*}
is a difference of two terms
\begin{align*}
	qS(q^{\nu}(n+h),m;c)e\left(\pm\frac{c\ell}{q}\right).
\end{align*}
The last factor is non-oscillating in the generic situation, as we only need to consider $c$ in the range $c\ll Qt^\varepsilon$ (and $\ell$ can be as small as $1$). Note that we are dropping the sum over $\ell$ as the object is essentially independent of $\ell$. Of course we need to execute the sum over $\ell$ trivially at the end. Also we will drop the factor $e(c\ell/q)$ from the expressions, as the sums over $\ell$ and $q$ are executed trivially at the end, and in our analysis below we will not apply any summation formula on the sum over $c$. In the expressions below always bear in mind that the $c$ sums have some arithmetic weight of size $1$, which does not depend on any other sums in the expressions.
So we continue our analysis with 
\begin{align*}
& \frac{N\sqrt{\ell}}{\sqrt{CQ}K^2}\sum_{n\sim N}\;\sum_{\nu=0}^\infty \lambda_F(q^\nu)\:\mathop{\sum}_{\substack{m\sim \tilde{N}/q^\nu \ell^2\\(q,m)=1}} \lambda_F(m)\\
\nonumber &\times \mathop{\sum}_{h\in\mathbb Z} (n+h)^{it}W\left(\frac{h}{H}\right)\:\sum_{\substack{c\sim C\\(c,q)=1}} \:S(n+h,mq^\nu;c)\:e\left(\frac{2\sqrt{mq^\nu (n+h)}}{c}\right).
\end{align*}
Now we execute the sum over $\nu$ by gluing $q^\nu$ back to $m$, and this yields the expression given in the lemma.
\end{proof}

\begin{remark}
At this point we can apply the Voronoi summation to get a bound which is satisfactory for small values of $C$. Indeed applying the Voronoi summation we get that \eqref{dual-od} is bounded by 
\begin{align*}
N^{1/2}HQ^2Kt^{1/3}\:\left\{\frac{N^{3/2}C}{(QK)^2t^{1/3}}+\frac{C}{Q}\right\}.
\end{align*}
The second term accounts for the diagonal contribution $m=n+h$ (after Voronoi) and the first term accounts for the other terms. The bound is satisfactory if
\begin{align}
\label{small-c-range-initial}
C\ll \min\left\{\frac{(QK)^2t^{1/3-\delta}}{N^{3/2}\ell},Qt^{-\delta}\right\}.
\end{align}
\end{remark}

\begin{remark}
If we pick $Q$ according to \eqref{pick-q} and $K$ according to \eqref{pick-k} then the condition \eqref{2nd-cond-QK} is satisfied if $\delta<1/20$.
\end{remark}

\bigskip

\section{Stationary phase analysis for dual off-diagonal}

\begin{Lemma}
\label{lemma:dual-od-after-poisson}
We have
$$
\mathcal{O}^\star(C)\ll \mathcal{O}^{\star \dagger}(C)
$$
where
\begin{align}
\label{dual-O-after-poisson}
&\mathcal{O}^{\star \dagger}(C)=\frac{NH\sqrt{\ell}}{\sqrt{CQ}K^2t^{1/6}}\:\mathop{\sum}_{m\sim \tilde{N}/\ell^2}\lambda_F(m)\\
\nonumber &\times \sum_{c\sim C}\mathop{\sum\sum}_{\substack{N<n\leq 2N\\ h\sim tC/N}} \:n^{it}e\left(\frac{2\sqrt{mn}}{c}-\frac{hn}{c}-\frac{\bar{h}m}{c}+\frac{A^2}{B}\right)\;W\left(\frac{2nA}{HB}\right),
\end{align} 
where 
$$
A=\frac{t}{2\pi}+\frac{\sqrt{mn}}{c}-\frac{hn}{c},\;\;\;B=\frac{t}{\pi}+\frac{\sqrt{mn}}{c}
$$
\end{Lemma}

\begin{proof}
Consider the sum over $h$ in \eqref{dual-od}, which is given by
\begin{align*}
\mathop{\sum}_{h\in\mathbb Z} (n+h)^{it}S(n+h,m;c)\:e\left(\frac{2\sqrt{m(n+h)}}{c}\right) W\left(\frac{h}{H}\right).
\end{align*}
We apply the Poisson summation formula with modulus $c$ to arrive at
\begin{align*}
	\frac{H}{c}\sum_{h\in\mathbb{Z}}\;\mathfrak{C}\;\mathfrak{I}
\end{align*}
where the character sum is given by
\begin{align*}
\mathfrak{C}=\sum_{b\bmod{c}}\;S(b+n,m;c)\: e\left(\frac{bh}{c}\right)=c\;e\left(-\frac{\bar{h}m}{c}-\frac{hn}{c}\right),
\end{align*}
and the integral is given by
\begin{align*}
\mathfrak{I}=\int W(y) e(f(y))\mathrm{d}y,
\end{align*}
where (we temporarily set)
\begin{align*}
f(y)= \frac{t}{2\pi}\log (n+Hy)+\frac{2\sqrt{m(n+Hy)}}{c}-\frac{Hhy}{c}.
\end{align*}
We have the Taylor expansion
\begin{align}
\label{f-dual}
f(y)= &\left[\frac{t}{2\pi}\log n+\frac{2\sqrt{mn}}{c}\right]+\frac{Hy}{n}\left[\frac{t}{2\pi}+\frac{\sqrt{mn}}{c}-\frac{hn}{c}\right]\\
\nonumber &-\frac{H^2y^2}{n^2}\left[\frac{t}{4\pi}+ \frac{\sqrt{mn}}{4c}\right]+O\left(1\right).
\end{align}
(The error term is a `flat' function in the sense that $y^jE^{(j)}(y)\ll 1$ with respect to all the variables.) Notice that the phase function is exactly similar to what we had in the direct off-diagonal, with the only difference that we have $c$ in place of $cq$. Applying the stationary phase expansion it follows that the dual off-diagonal is given by \eqref{dual-O-after-poisson}.
\end{proof}

The weight function in \eqref{dual-O-after-poisson} implies that we have $A \asymp t^{2/3}$. This can be used to conclude the following restrictions 
\begin{align}
\label{h}
	D:=\frac{hn}{c}-\frac{t}{2\pi }\ll  t\:\max\left\{\frac{1}{t^{1/3}},\frac{QK^2}{C\ell t}\right\}=:t\Delta
\end{align}
and 
\begin{align}
\label{m-res-size}
\left|m-\frac{D^2c^2}{n}\right|\ll  \frac{Q^2K^4}{N\ell^2}\;\min\left\{1, \frac{C\ell t^{2/3}}{QK^2}\right\}=\frac{\tilde{N}}{\ell^2}\:\frac{1}{t^{1/3}\Delta}.
\end{align}
We derive that \eqref{dual-O-after-poisson} is  bounded by 
\begin{align}
\label{basic-bound-before-cauchy}
\sqrt{N}HQ^2Kt^{1/3}\;\frac{C^{3/2}Kt^{1/6}}{\sqrt{NQ}\ell^{3/2}},
\end{align}
which is not sufficient for our purpose. However this bound is fine for smaller values of $C$, namely in the range
\begin{align}
\label{small-c-range}
C\ll \frac{(NQ)^{1/3}t^{-2\delta/3}\ell^{1/3}}{K^{2/3}t^{1/9}}.
\end{align}
Below we proceed with $C$ which lies in the range complementary to both \eqref{small-c-range-initial} and \eqref{small-c-range}.
 \\

\section{Cauchy for dual off-diagonal}
\label{cauchy-2-sec}
We apply the Cauchy inequality to bound \eqref{dual-O-after-poisson} by
\begin{align}
\label{ineq-omega}
\frac{N^{3/2}H\sqrt{\ell}}{\sqrt{CQ}K^2t^{1/6}}\;\Omega_\ell^{1/2}
\end{align}
where now $\Omega_\ell$ is given by
\begin{align}
\label{omega-2}
\sum_{\substack{n\sim N}}\;&\Bigl|\mathop{\sum}_{m\sim \tilde{N}/\ell^2}\lambda_F(m)\sum_{c\sim C}\mathop{\sum}_{\substack{h\sim tC/N}} e\left(\frac{2\sqrt{mn}}{c}-\frac{hn}{c}-\frac{\bar{h}m}{c}+\frac{A^2}{B}\right)\;W\left(\frac{2nA}{HB}\right)\Bigr|^2.
\end{align}
We open the absolute square to arrive at 
\begin{align}
\label{omega-22}
 &\sum_{n\in \mathbb{Z}} W\left(\frac{n}{N}\right)\:\mathop{\sum\sum}_{m_1,m_2\sim \tilde{N}/\ell^2}\;\mathop{\sum\sum}_{h_1,h_2\sim tC/N}\mathop{\sum\sum}_{c_1,c_2\sim C} \:\lambda_F(m_1)\lambda_F(m_2)\\
\nonumber &\times e\left(-\frac{nh_1}{c_1}+\frac{nh_2}{c_2}+\frac{2\sqrt{m_1n}}{c_1}-\frac{2\sqrt{m_2n}}{c_2}+\frac{A_1^2}{B_1}-\frac{A_2^2}{B_2}\right)\\
\nonumber &\times  e\left(\frac{m_2\bar{h}_2}{c_2}-\frac{m_1\bar{h}_1}{c_1}\right)W\left(\frac{2nA_1}{HB_1}\right)W\left(\frac{2nA_2}{HB_2}\right)
\end{align}
where the subscript in $A_1$ indicates that the related parameters are $(m_1,h_1,c_1)$ and so on. 

\begin{Lemma}
\label{lemma:large-gap}
Suppose $Q$ is given by \eqref{pick-q}. The contribution of the terms with
\begin{align}
\label{small-gap-condition}
m_1c_2^2-m_2c_1^2\gg t^\varepsilon \frac{\tilde{N}C^2}{\ell^2}\frac{C\ell t\Delta^2}{QK^2}\asymp t^\varepsilon \frac{QK^2C^3\ell t\Delta^2}{N\ell^2}
\end{align}
to \eqref{omega-22} is negligibly small.
\end{Lemma}

\begin{proof}
The weight functions in \eqref{omega-22} imply that 
\begin{align}
\label{h-c-rest}
\left|\frac{h_1}{c_1}-\frac{h_2}{c_2}\right|\ll \frac{t\Delta}{N}.
\end{align}
Applying the Poisson summation formula the $n$ sum transforms into
\begin{align}
\label{hi}
N\Delta\:\sum_{\substack{n\in\mathbb{Z}\\ h_1c_2-h_2c_1\equiv n\bmod{c_1c_2}}}\;\mathfrak{I}
\end{align}
where the integral is given by
\begin{align}
\mathfrak{I}= \int V(y) e\left(a_1\frac{N\Delta y}{n_0}+a_2\frac{N^2\Delta^2 y^2}{n_0^2}+\dots\right)\mathrm{d}y
\end{align}
with $n_0=tc_1/2\pi h_1$, $V^{(j)}(y)\ll (t^{1/3}\Delta)^j$,
$$
a_1=-\frac{\sqrt{n_0}}{4}\left(\frac{\sqrt{m_1}}{c_1}-\frac{\sqrt{m_2}}{c_2}\right)+n_0\frac{c_1h_2}{c_2h_1}\left(\frac{h_1}{c_1}-\frac{h_2}{c_2}\right)-\frac{nn_0}{c_1c_2}+O\left(\frac{QK^2}{C\ell t^{2/3}}\right)
$$
and
\begin{align*}
a_2=&\sqrt{n_0}\left(\frac{\sqrt{m_1}}{c_1}-\frac{\sqrt{m_2}}{c_2}\right)+\left(\frac{\sqrt{m_1n_0}}{2c_1}-\frac{h_1n_0}{c_1}\right)^2\left(\frac{t}{\pi}+\frac{\sqrt{m_1n_0}}{c_1}\right)^{-1}
\\
&-\left(\frac{\sqrt{m_2n_0}}{2c_2}-\frac{h_2n_0}{c_2}\right)^2\left(\frac{t}{\pi}+\frac{\sqrt{m_2n_0}}{c_2}\right)^{-1}+\text{smaller order terms},
\end{align*}
and so on. Using the congruence condition in \eqref{hi} we write
$$
n=h_1c_2-h_2c_1+\mu c_1c_2,
$$
and applying Lemma~\ref{lemma:small-int}, we get that the integral is negligibly small if 
\begin{align*}
\left|\frac{\sqrt{n_0}}{4}\left(\frac{\sqrt{m_1}}{c_1}-\frac{\sqrt{m_2}}{c_2}\right)+n_0\mu\right|\gg t\Delta^2.
\end{align*}
The case of $\mu\neq 0$ is easily ruled out as then we would need $C\ll QK^2/N\ell$, which can not happen due to \eqref{small-c-range-initial} and \eqref{small-c-range}, if we pick $Q$ as in \eqref{pick-q}.
Now for $\mu=0$ the integral is negligibly small due to the condition  \eqref{small-gap-condition}. The lemma follows.
\end{proof}

Note that in the range complementary to \eqref{small-gap-condition} we get an extra saving saving of $Q^{1/2}K/(C\ell)^{1/2}t^{1/2}\Delta$, at the price of loosing the restriction \eqref{m-res-size} on one of the $m_i$'s. So effectively we save $Q^{1/2}K/(C\ell)^{1/2}t^{2/3}\Delta^{3/2}$ which is not enough for our purpose, as the resulting bound is 
\begin{align}
\label{basic-bound-after-cauchy}
\sqrt{N}HQ^2Kt^{1/3}\;\frac{C^2t^{5/6}\Delta^{3/2}}{\sqrt{N}Q\ell}.
\end{align}
in place of \eqref{basic-bound-before-cauchy}. 

\bigskip

\section{Second application of Cauchy on dual off-diagonal}
\label{2nd-cauchy}

Now we consider \eqref{omega-22} with the restriction complementing  \eqref{small-gap-condition}, i.e. terms with `small gap'. This can be dominated by
\begin{align}
\label{omega-22-dominant}
\Theta_\ell :=\mathop{\sum\sum}_{m_1,m_2\sim \tilde{N}/\ell^2}\;&\Bigl|\sum_{n\sim N} \mathop{\sum\sum}_{h_1,h_2\sim tC/N}\mathop{\sum\sum}_{c_1,c_2\sim C} e\left(-\frac{nh_1}{c_1}+\frac{nh_2}{c_2}\right)\\
\nonumber \times  &e\left(\frac{m_2\bar{h}_2}{c_2}-\frac{m_1\bar{h}_1}{c_1}\right)
e\left(\frac{2\sqrt{m_1n}}{c_1}-\frac{2\sqrt{m_2n}}{c_2}+\frac{A_1^2}{B_1}-\frac{A_2^2}{B_2}\right)\\
\nonumber \times  &W\left(\frac{2nA_1}{HB_1}\right)W\left(\frac{2nA_2}{HB_2}\right)V\left(\frac{m_1c_2^2-m_2c_1^2}{N_0}\right)\Bigr|,
\end{align}
where $N_0=QK^2C^3t\Delta^2/N\ell$. The trivial bound for this sum is given by
\begin{align}
\label{trivial-bound-before-cauchy}
\left(\frac{\tilde{N}}{\ell^2}\frac{1}{t^{1/3}\Delta}\right)\;\left(\frac{\tilde{N}}{\ell^2}\frac{C\ell t\Delta^2}{QK^2}\right)\;N\Delta\;\frac{t^{2}C^2\Delta}{N^2}\;C^2\;\asymp \left(\frac{Q^{3/2}K^3 C^{5/2} t^{4/3}\Delta^{3/2}}{N^{3/2}\ell^{3/2}}\right)^2,
\end{align}
which when substituted for $\Omega_\ell$ in \eqref{ineq-omega} yields the bound \eqref{basic-bound-after-cauchy}.\\

\begin{Lemma}
\label{lemma:2nd-cauchy}
We have
$$
\Theta_\ell\ll \frac{\tilde{N}^2}{\ell^4}\;\Xi_\ell^{1/2}
$$
where
\begin{align}
\label{m1m2-cong}
\Xi_\ell=\mathop{\sum\sum}_{n,n'\sim N}\;\mathop{ \mathop{\sum\sum\sum\sum}_{h_1,h_1',h_2,h_2'\sim tC/N}\;\mathop{\sum\sum\sum\sum}_{c_1,c_1',c_2,c_2'\sim C}\;\mathop{\sum\sum}_{\substack{m_1,m_2\in\mathbb{Z}}}}_{\substack{\bar{h}_1c_1'-\bar{h}'_1c_1\equiv m_1\bmod{c_1c_1'}\\ \bar{h}_2c_2'-\bar{h}'_2c_2\equiv m_2\bmod{c_2c_2'}}}\;\mathfrak{I},
\end{align}
and 
\begin{align*}
\mathfrak{I}=&\iint e\left(F_1(y_1)-F_2(y_2)\right)W\left(\frac{2nA_1}{HB_1}\right)W\left(\frac{2n'A_1'}{HB_1'}\right)W\left(\frac{2nA_2}{HB_2}\right)W\left(\frac{2n'A_2'}{HB_2'}\right)\\
&\times W(y_1)W(y_2)V\left(\frac{y_1c_2^2-y_2c_1^2}{N_0\ell^2/\tilde{N}}\right)V\left(\frac{y_1c_2^{'2}-y_2c_1^{'2}}{N_0\ell^2/\tilde{N}}\right)\mathrm{d}y_1\mathrm{d}y_2,
\end{align*}
with 
$$
F_i(u)=\frac{2\sqrt{\tilde{N}un}}{c_i\ell}-\frac{2\sqrt{\tilde{N}un'}}{c_i'\ell}+\frac{A_i^2}{B_i}-\frac{A_i^{'2}}{B_i'}-\frac{\tilde{N}m_iu}{c_ic_i'\ell^2}.
$$
(Here in $A_i$ etc. $m_i$ is replaced by $\tilde{N}u/\ell^2$.)
\end{Lemma}

\begin{proof}
We apply the Cauchy inequality (yet again) to \eqref{omega-22-dominant}, and then open the absolute square and apply the Poisson summation on $(m_1,m_2)$. The resulting character sums give rise to the congruence conditions, and the Fourier transform becomes $\mathfrak{I}$ after a change of variables. \end{proof}

Now we will study the integral in some detail. \\

\begin{Lemma}
\label{lemma:int-1}
The integral is negligibly small if
\begin{align}
\label{c-restriction}
c_1c_2'-c_1'c_2\gg C^2\:\frac{C\ell t\Delta^2}{QK^2}.
\end{align}
or if
\begin{align}
\label{m1-m2-rest}
\max\{m_1, m_2\}\gg t^\varepsilon \frac{NC\ell}{QK^2},\;\;\;\text{or}\;\;\;
m_1-m_2\gg  t^\varepsilon \frac{(C\ell)^2t\Delta^2}{\tilde{N}}.
\end{align}
\end{Lemma}

\begin{proof}
The condition \eqref{c-restriction} follows by considering the last two weights in the integral. Also by repeated integration by parts we see that the integral is negligibly small if
\begin{align*}
\max\{m_1, m_2\}\gg t^\varepsilon \frac{NC\ell}{QK^2}.
\end{align*}
Now we set $w=(y_1c_2^2-y_2c_1^2)\tilde{N}/N_0\ell^2$, $\alpha=c_2^2/c_1^2$ and $\Theta=C^3\ell t\Delta^2/c_1^2QK^2$. Then substituting for $y_2$, and using the Taylor expansion, we arrive at the expression 
\begin{align}
\label{int-taylor}
\mathfrak{I}=\frac{C\ell t\Delta^2}{QK^2}\iint e\left(F_1(y_1)-F_2(\alpha y_1)+\Theta wF_2'(\alpha y_1)-\dots\right)U\left(y_1,w\right)\mathrm{d}y_1\mathrm{d}w,
\end{align}
where the weight function satisfies $U^{(j_1,j_2)}\ll (t^{1/3}\Delta)^{j_1+j_2}$.
 Now by repeated integration by parts with respect to the $y_1$ we see that the integral is negligibly small if
\begin{align*}
m_1c_1c_2'-m_2c_2c_1'\gg \frac{C^4\ell^2 t\Delta^2}{\tilde{N}}.
\end{align*}
This condition reduces to the second condition in \eqref{m1-m2-rest} because of \eqref{c-restriction}. 
\end{proof}

We can say more about the size of the integral if $(m_1,m_2)\neq (0,0)$. Let us assume that $m_2\neq 0$.\\

\begin{Lemma}
\label{lemma:int-2}
Suppose $m_2\neq 0$ then we have
\begin{align}
\label{J-int}
\mathfrak{I}\ll  \frac{1}{|m_2|}\:\frac{(C\ell)^{5/2} Nt^{2/3}\Delta}{(QK^2)^{5/2}}.
\end{align}
\end{Lemma}

\begin{proof}
Consider the form of the integral given in \eqref{int-taylor}.
Given $y_1$, look at the integral over $w$, which turns out to be negligibly small due to Lemma~\ref{lemma:small-int} if
$$
\Theta |F_2'(\alpha y_1)|\gg \Theta \sqrt{\frac{QK^2}{C\ell}}+t^{1/3}\Delta.
$$
This boils down to the condition
\begin{align}
\label{star}
|y_1-\star|\gg t^\varepsilon\frac{1}{|m_2|}\:\frac{C\ell N}{QK^2}\:\left(\frac{1}{t^{2/3}\Delta}+\sqrt{\frac{C\ell}{QK^2}}\right)\asymp t^\varepsilon\frac{1}{|m_2|}\:\frac{(C\ell)^{3/2} N}{(QK^2)^{3/2}},
\end{align}
for some $\star$.
In generic case this cuts down the length of the $y_1$ integral by $K$ which will be taken to be slightly smaller than $t^{1/3}$. More precisely we get
\begin{align*}
\mathfrak{I}\ll \frac{1}{t^{1/3}\Delta}\;\frac{C\ell t\Delta^2}{QK^2}\:\frac{1}{|m_2|}\:\frac{(C\ell)^{3/2} N}{(QK^2)^{3/2}}.
\end{align*}
Note that the extra saving of $t^{1/3}\Delta$ comes by considering the weights in the $w$ integral. 
\end{proof}

The explicit expression for $\star$ in \eqref{star} is not required in our analysis. We just note that it follows from the fact $\star \sim 1$ that we have
\begin{align*}
\frac{\sqrt{n}}{c_2}-\frac{\sqrt{n'}}{c_2'}\ll \frac{M_2QK^2}{C^2\ell \sqrt{N}}+\frac{\sqrt{N}}{Ct^{1/3}}.
\end{align*}
For $M_2$ of the generic size, i.e. $NC\ell/QK^2$, the condition does not impose any restriction. But for $M_2$ smaller it squeezes the variables together. Indeed combining with the condition \eqref{h} it follows that  
\begin{align}
\label{hc-distance}
h_2c_2-h_2'c_2'\ll \frac{tC^2}{N}\left(\frac{M_2QK^2}{NC\ell}+\Delta\right).
\end{align}
Furthermore in case $QK^2/C\ell t\gg 1/t^{1/3}$, we observe that the weights involved in the integral imply that 
\begin{align*}
\left|\frac{D_i^2c_i^2}{n}-\frac{D_i^{'2}c_i^{'2}}{n'}\right|\ll  \frac{\tilde{N}}{\ell^2}\:\frac{1}{t^{1/3}\Delta},
\end{align*}
which in turn boils down to saving $t^{1/3}\Delta$ in the count for $(n,n')$.
\\

\section{The zero frequency}

Now our strategy is to break the sum \eqref{m1m2-cong} into blocks according to the sizes of $m_i$. We will estimate the size of each such block and thus get an estimate for its contribution to $\Theta_\ell$ (as given in Lemma~\ref{lemma:2nd-cauchy}). We will say that our estimate is `satisfactory' if its contribution to \eqref{ineq-omega} is bounded by
$$
O(\sqrt{N}HKQ^2 t^{1/3-\delta}).
$$
We begin by considering the contribution of the zero frequency $(m_1,m_2)=(0,0)$. \\

\begin{Lemma}
\label{lemma:zero}
Let $Q$ be as given in \eqref{pick-q}. The contribution of $(m_1,m_2)=(0,0)$ to the dual off-diagonal is satisfactory if 
\begin{align}
\label{upper-bound-on-K}
K\ll \min\{t^{1/3-2\delta},Nt^{-1/3-2\delta}\},
\end{align} 
and $\delta<1/12$.
\end{Lemma}

\begin{proof}
For $m_1=m_2=0$ the congruence conditions in \eqref{m1m2-cong} imply that $c_1=c_1'$ and $c_2=c_2'$. Also $h_i\equiv h_i'\bmod{c_i}$, so that $h_i-h_i'=c_ig_i$ for some integer $g_i$. It follows by repeated integration by parts w.r.t. $y_1$, that the integral $\mathfrak{I}$ is negligibly small 
in the range 
$$
n-n'\gg Nt^\varepsilon/t^{1/3}.
$$
Combining with the existing restriction \eqref{h}, namely $|n-tc_i/2\pi h_i|\ll N\Delta$ it follows that 
$$
h_i-h_i'\ll \frac{tC}{N}\Delta.
$$
So the number of $g_i$ is given by $O(1+t\Delta/N)=O(1+t^{2/3}/N)$, where for the last inequality we use the fact that $QK^2\ll C\ell N$ (which follows as $C$ is in the complementary range to \eqref{small-c-range-initial}, and $Q$ is picked according to \eqref{pick-q}). Finally we observe that \eqref{small-gap-condition} together with \eqref{m-res-size} imply the stronger restriction
\begin{align}
\label{h-c-strong}
\left|\frac{h_1}{c_1}-\frac{h_2}{c_2}\right|\ll \frac{t^{2/3}}{N}+\frac{t\Delta^2}{N}
\end{align}
in place of \eqref{h-c-rest}. It follows that the number of $(h_1,h_2,c_1,c_2)$ is bounded by 
$$
\frac{tC^2}{N}\:\left(1+\frac{C^2t^{2/3}}{N}+\frac{C^2t\Delta^2}{N}\right).
$$
Hence, using the bound 
$$
\mathfrak{I}\ll \frac{1}{t^{1/3}\Delta}\;\frac{C\ell t\Delta^2}{QK^2},
$$ 
we get that the contribution of the zero frequency to \eqref{omega-22-dominant} is bounded by
\begin{align*}
\frac{\tilde{N}^2}{\ell^4}\:\left\{\frac{C\ell t^{2/3}\Delta}{QK^2}\:\frac{N^2\Delta}{t^{1/3}}\:\left(1+\frac{t^{2/3}}{N}\right)^2 \;\frac{tC^2}{N}\:\left(1+\frac{C^2t^{2/3}}{N}+\frac{C^2t\Delta^2}{N}\right)\right\}^{1/2}.
\end{align*}
Substituting this bound for $\Omega_\ell$ in \eqref{ineq-omega}, trivially summing over $\ell$ and using the inequality $C\ell \Delta\ll Qt^{-1/3}$ (as $K\ll t^{1/3}$ by choice), it follows that the overall contribution of the zero frequency to the dual off-diagonal is given by
\begin{align}
\label{basic-bound-after-cauchy-0-freq}
\sqrt{N}HQ^2Kt^{1/3}\;\frac{K^{1/2}N^{1/4}}{(CQ)^{1/4}t^{1/3}}\;\left(1+\frac{t^{1/3}}{\sqrt{N}}\right) \;\left(1+\frac{C^{1/2}t^{1/6}}{N^{1/4}}+\frac{Q^{1/2}t^{1/12}}{N^{1/4}}\right).
\end{align}
The main term in this contribution is 
\begin{align*}
\sqrt{N}HQ^2Kt^{1/3}\;\frac{K^{1/2}N^{1/4}}{(CQ)^{1/4}t^{1/3}}\;\left(1+\frac{t^{1/3}}{\sqrt{N}}\right) \;\frac{C^{1/2}t^{1/6}}{N^{1/4}},
\end{align*}
which is satisfactory under the condition given in \eqref{upper-bound-on-K} (recall that $C\ll Q$).
For the other terms we use the fact that $C$ is in the complementary range to \eqref{small-c-range-initial} (resp. \eqref{small-c-range}) for $N\ll t^{2/3}$ (resp. $N\gg t^{2/3}$). The overall contribution of these terms is satisfactory if we pick $Q$ according to \eqref{pick-q}, and take $\delta<1/12$.
\end{proof}

\begin{remark}
We will pick 
\begin{align}
\label{pick-k}
K=\begin{cases}
t^{1/3-2\delta} &\text{if $t^{2/3}<N\ll t^{1+\varepsilon}$}\\
Nt^{-1/3-2\delta} &\text{if $t^{2/3-2\delta}<N\leq t^{2/3}$.}
\end{cases}
\end{align}
\end{remark}

\bigskip

\section{The final counting problem}
\label{counting-problem}

Now we will analyze the contribution of the non-zero frequencies $(m_1,m_2)\neq (0,0)$. We have reduced the problem to counting the number of 
\begin{align}
h_i, h_i'\sim tC/N,\;\;\; c_i, c_i'\sim C,\;\;\;m_i\sim M_i\ll C\ell N/QK^2,\;\;\;i=1,2,
\end{align}
satisfying the following constraints  
\begin{align}
\label{counting-problem-final}
&\left|\frac{h_1}{c_1}-\frac{h_2}{c_2}\right|,\;\left|\frac{h_1'}{c_1'}-\frac{h_2'}{c_2'}\right|\ll \frac{t\Delta}{N},\;\;\;
\left|\frac{c_1'}{c_1}-\frac{c_2'}{c_2}\right|\ll \frac{C\ell t\Delta^2}{QK^2}\\
\nonumber & \bar{h}_1c_1'-\bar{h}'_1c_1\equiv m_1\bmod{c_1c_1'},\;\;\;\bar{h}_2c_2'-\bar{h}'_2c_2\equiv m_2\bmod{c_2c_2'},
\end{align}
with $C\ell\ll Q$ and $m_1-m_2\ll (C\ell)^2t\Delta^2/\tilde{N}$. We set $\mathbf{v}_i=(h_i,h_i',c_i,c_i')$. Let 
$$
N(C,M_1,M_2)=\#\{\mathbf{v}_1,\mathbf{v}_2,m_1,m_2: \:\text{satisfying \eqref{counting-problem-final}} \}.
$$ 
Suppose $m_2\neq 0$, then the contribution of the $m_2\sim M_2$ block to \eqref{omega-22-dominant} is bounded by
\begin{align}
\label{omega-22-dominant-new}
 &\frac{\tilde{N}^2}{\ell^4}\;\left[\frac{1}{M_2}\:\frac{(C\ell)^{5/2} Nt^{2/3}\Delta}{(QK^2)^{5/2}}\:\frac{(N\Delta)^2}{t^{1/3}\Delta}\:N(C,M_1,M_2)\right]^{1/2},
 \end{align}
 as we have $(N\Delta)^2/t^{1/3}\Delta$ many pairs $(n,n')$. Let $(c_i,c_i')=\delta_i$, for $i=1,2$, and let $(c_1/\delta_1,c_2/\delta_2)=\delta_3$. We write $c_i=\delta_i\delta_3 d_i$ and $c_i'=\delta_i d_i'$ for $i=1,2$. It follows that  $\delta_i|m_i$ and we write $m_i=\delta_i\mu_i$ for $i=1,2$. Also it follows that $\delta_1\delta_2\delta_3|(c_1c_2'-c_1'c_2)$, and we set $(c_1c_2'-c_1'c_2)=\delta_1\delta_2\delta_3u$ with $u\ll C^3\ell t\Delta^2/\delta_1\delta_2\delta_3QK^2$. Let $N_0$ (resp. $N_{\neq 0}$) denote the contribution of $u=0$ (resp. $u\neq 0$) and  $\delta_i\sim \Delta_i$ to $N(C,M_1,M_2)$.\\

\begin{Lemma}
We have
\begin{align*}
    N_0\ll M_2\frac{t^{2}C^2}{N^2}\: \left(1+\frac{t^{1/3}N}{\Delta_1K^4}\right)\left(1+\frac{t\Delta\Delta_1}{N}\right)\left(1+\frac{t\Delta}{N}\right).
\end{align*}
\end{Lemma}

\begin{proof}
For $u=0$ we get $d_1=d_2$ and $d_1'=d_2'$. Next $h_2$ and $h_2'$ are determined using the last condition in \eqref{counting-problem-final}, which translates to
$$
\bar{h}_2d_2'-\bar{h}'_2d_2\delta_3\equiv \mu_2\bmod{d_2d_2'\delta_2\delta_3}
$$ 
It follows that there are $O(t^2\delta_2(\delta_2,d_2\delta_3)/N^2)$ many pairs $(h_2, h_2')$. Now we write $h_1=[h_2c_1/c_2]+g$ and $h_1'=[h_2'c_1'/c_2']+g'$, with $g, g'\ll 1+Ct\Delta /N$ from the first two conditions in \eqref{counting-problem-final}. Finally the number of $g, g'$ satisfying the fourth condition in \eqref{counting-problem-final} is given by  $O((\delta_1,d_1')(1+\delta_1t\Delta/N)(1+t\Delta/N))$. It follows that $N_0$ is dominated by
\begin{align*}
\mathop{\sum\sum\sum}_{\substack{\delta_i\sim \Delta_i}}\mathop{\sum\sum}_{\substack{d_1'\sim C/\Delta_1\\ d_2\sim C/\Delta_2\Delta_3}}\:\mathop{\sum\sum}_{\substack{\mu_i\sim M_i/\Delta_i\\\delta_1\mu_1-\delta_2\mu_2\ll \frac{(C\ell)^2t\Delta^2}{\tilde{N}}}}\:&\frac{t^2}{N^2}\;\delta_2 (\delta_1,d_1')(\delta_2,d_2\delta_3)\\
&\times \left(1+\frac{t\Delta\Delta_1}{N}\right)\left(1+\frac{t\Delta}{N}\right).
\end{align*}
Summing over $\mu_i$, and gluing $\delta_3$ with $d_2$, we arrive at 
\begin{align*}
\mathop{\sum\sum}_{\substack{\delta_i\sim \Delta_i}}&\mathop{\sum\sum}_{\substack{d_1'\sim C/\Delta_1\\ d_2\sim C/\Delta_2}}\:\frac{t^2}{N^2}\;M_2 (\delta_1,d_1')(\delta_2,d_2)\:\left(1+\frac{(C\ell)^2t\Delta^2}{\delta_1\tilde{N}}\right)\left(1+\frac{t\Delta\Delta_1}{N}\right)\left(1+\frac{t\Delta}{N}\right).
\end{align*}
Finally executing the sums over the remaining variables we arrive at
\begin{align*}
\frac{t^2M_2C^2}{N^2}\: \left(1+\frac{(C\ell)^2t\Delta^2}{\Delta_1\tilde{N}}\right)\left(1+\frac{t\Delta\Delta_1}{N}\right)\left(1+\frac{t\Delta}{N}\right).
\end{align*}
The lemma now follows by using the fact that $C\ell \Delta\ll Qt^{-1/3}$.
\end{proof}

\begin{Lemma}
\label{lemma:zero-frequency-contribution}
The contribution of $u=0$ to the dual off-diagonal is satisfactory if $\delta<1/126$.
\end{Lemma}

\begin{proof}
Substituting the bound for $N_0$ from the previous lemma in place of $N(\dots)$ in \eqref{omega-22-dominant-new}, and then substituting the resulting bound in place of $\Omega_\ell$ in \eqref{ineq-omega} we obtain
\begin{align*}
    \sqrt{N}HQ^2Kt^{1/3}\frac{C^{1/8}N^{1/4}}{K^{1/4}Q^{5/8}t^{1/12}\ell^{3/2}}\left(1+\frac{t^{1/12}N^{1/4}}{\Delta_1^{1/4}K}\right)\left(1+\frac{(t\Delta\Delta_1)^{1/4}}{N^{1/4}}\right)\left(1+\frac{(t\Delta)^{1/4}}{N^{1/4}}\right).
\end{align*}
(The term $\ell^{3/2}$ in the denominator will be responsible for the convergence of the sum over $\ell$. As such we are going to ignore the $\ell$ terms below.)
Once we multiply the terms out we get eight terms, two of which have $\Delta_1$ on the numerator. We consider these two terms first, namely
\begin{align*}
    \frac{C^{1/8}N^{1/4}}{K^{1/4}Q^{5/8}t^{1/12}}\frac{(t\Delta\Delta_1)^{1/4}}{N^{1/4}}\left(1+\frac{(t\Delta)^{1/4}}{N^{1/4}}\right).
\end{align*}
Using the fact that $\Delta_1\ll C$, and that $C\Delta\ll Qt^{-1/3}$ we see that the above expression is bounded by
\begin{align*}
    \frac{t^{1/12}C^{1/8}}{K^{1/4}Q^{3/8}}\left(1+\frac{(t\Delta)^{1/4}}{N^{1/4}}\right)\ll \frac{t^{1/12}C^{1/8}}{K^{1/4}Q^{3/8}}\left(1+\frac{t^{1/6}}{N^{1/4}}+\frac{Q^{1/4}K^{1/2}}{(CN)^{1/4}}\right).
\end{align*}
The last expression is dominated by $t^{-\delta}$ if  $\delta<1/126$, when we pick $Q$ according to \eqref{pick-q} and $K$ according to \eqref{pick-k}. Next we consider terms which do not have $\Delta_1$ in the numerator. In this case we can as well take $\Delta_1=1$. As such we need to consider the expression
\begin{align}
\label{in-place}
    \frac{C^{1/8}N^{1/4}}{K^{1/4}Q^{5/8}t^{1/12}}\left(1+\frac{t^{1/12}N^{1/4}}{K}\right)\left(1+\frac{(t\Delta)^{1/2}}{N^{1/2}}\right).
\end{align}
Now using the fact 
$$
\frac{t\Delta}{N}\ll \frac{t^{2/3}}{N}+\frac{QK^2}{CN} \ll \frac{t^{2/3}}{N}+\frac{N^{1/2}}{Qt^{1/3-\delta}}\ll \frac{t^{2/3}}{N}+\frac{K^{6/5}}{N^{1/10}t^{2/5-3\delta/5}}\ll \frac{t^{2/3}}{N}+1,
$$
we get that the contribution of this last expression is also dominated by $t^{-\delta}$ if $\delta<1/126$.
\end{proof}

\begin{remark}
It should be possible to improve our estimate by exploiting the fact that in most cases we would not have any $g$, $g'$. However for the purpose of this paper the above estimate is satisfactory. 
\end{remark}

\begin{Lemma}
\label{lemma:counting-problem}
We have
\begin{align*}
N_{\neq 0}\ll M_2\frac{t^{7/3}C^3Q}{\ell \Delta_1\Delta_2\Delta_3 N^2K^2}\: \left(1+\frac{t^{1/3}N}{\Delta_1K^4}\right)\left(1+\frac{t\Delta\Delta_1}{N}\right)\left(1+\frac{t\Delta}{N}\right).
\end{align*}
\end{Lemma}

\begin{proof}
Given $\delta_1,\delta_2,\delta_3, d_1,d_2, \mu_1, \mu_2, u$, we now determine the remaining variables. First $d_1'$ and $d_2'$ are obtained using 
$$
\d_1'd_2-d_1d_2'=u
$$ 
Once $d_1'$ is determined there is at most one choice for $d_2'$. Writing the equation as a congruence modulo $d_1$, we see that there are $O(t^\varepsilon)$ many $d_1'$, and hence those many pairs $(d_1',d_2')$. Next $h_2$ and $h_2'$ are determined using the last condition in \eqref{counting-problem-final}, which translates to
$$
\bar{h}_2d_2'-\bar{h}'_2d_2\delta_3\equiv \mu_2\bmod{d_2d_2'\delta_2\delta_3}
$$ 
It follows that there are $O(t^2\delta_2(\delta_2,d_2\delta_3)/N^2)$ many pairs $(h_2, h_2')$. Now we write $h_1=[h_2c_1/c_2]+g$ and $h_1'=[h_2'c_1'/c_2']+g'$, with $g, g'\ll 1+Ct\Delta /N$ from the first two conditions in \eqref{counting-problem-final}. Finally the number of $g, g'$ satisfying the fourth condition in \eqref{counting-problem-final} is given by $O((\delta_1,d_1)(1+\delta_1t\Delta/N)(1+t\Delta/N))$. It follows that
\begin{align}
\label{n0}
N_{\neq 0}\ll \mathop{\sum\sum\sum}_{\delta_i\sim \Delta_i}&\mathop{\sum\sum}_{d_i\sim C/\Delta_i\Delta_3}\:\mathop{\sum\sum}_{\substack{\mu_i\sim M_i/\Delta_i\\\delta_1\mu_1-\delta_2\mu_2\ll \frac{(C\ell)^2t\Delta^2}{\tilde{N}}}}\:\sum_{0<u\ll \frac{C^3\ell t\Delta^2}{\Delta_1\Delta_2\Delta_3 QK^2}}\frac{t^2}{N^2}\\
\nonumber &\times \delta_2\delta_3 (\delta_1,d_1\delta_3)(\delta_2,d_2\delta_3)\:\left(1+\frac{t\Delta\Delta_1}{N}\right)\left(1+\frac{t\Delta}{N}\right).
\end{align}
Summing over $\mu_i$ and $u$ we see that the above expression  is dominated by
\begin{align*}
 \mathop{\sum\sum\sum}_{\delta_i\sim \Delta_i}&\mathop{\sum\sum}_{d_i\sim C/\Delta_i\Delta_3}\:\frac{C^3\ell t^3\Delta^2}{N^2\Delta_1 QK^2}\;\frac{M_2\Delta_3}{\delta_2}\left(1+\frac{(C\ell)^2t\Delta^2}{\delta_1\tilde{N}}\right) \\
&\times  (\delta_1,d_1\delta_3)(\delta_2,d_2\delta_3)\:\left(1+\frac{t\Delta\Delta_1}{N}\right)\left(1+\frac{t\Delta}{N}\right).
\end{align*}
Then summing over $\delta_1$, $\delta_2$ we arrive at
\begin{align*}
 \sum_{\delta_3\sim \Delta_3}\mathop{\sum\sum}_{d_i\sim C/\Delta_i\Delta_3}\:\frac{C^3\ell t^3\Delta^2M_2\Delta_3}{N^2QK^2}\;\left(1+\frac{(C\ell)^2t\Delta^2}{\Delta_1\tilde{N}}\right) \left(1+\frac{t\Delta\Delta_1}{N}\right)\left(1+\frac{t\Delta}{N}\right).
\end{align*}
This is seen to be dominated by 
\begin{align*}
\frac{C^5\ell t^3\Delta^2M_2}{\Delta_1\Delta_2\Delta_3 N^2QK^2}\;\left(1+\frac{(C\ell)^2t\Delta^2}{\Delta_1\tilde{N}}\right) \left(1+\frac{t\Delta\Delta_1}{N}\right)\left(1+\frac{t\Delta}{N}\right).
\end{align*}
  The lemma follows by using the fact that $C\ell \Delta\ll Qt^{-1/3}$.
\end{proof}

\begin{Lemma}
The contribution of $u\neq 0$ to the dual off-diagonal is satisfactory if $\delta<1/126$ and if either $N\ll t^{1-10\delta}$ or $\Delta_1\Delta_2\Delta_3\gg t^{18\delta}$.
\end{Lemma}

\begin{proof}
First compare the bounds for $N_0$ and $N_{\neq 0}$. The later is obtained from the former by multiplying with the factor $t^{1/3}CQ/K^2$. We observe that $t^{1/3}Q/K^2\ll 1$ for $\delta<1/126$. Now in the proof of Lemma~\ref{lemma:zero-frequency-contribution}, in the terms with $\Delta_1$ in the numerator, we replaced $\Delta_1$ by $C$. In $N_{\neq 0}$ the $\Delta_1$ in the numerator is balanced by the extra factor in the denominator, and so the contribution of these terms is balanced as $t^{1/3}Q/K^2\ll 1$. So we are led to consider the expression
\begin{align}
\label{generic-count}
    \frac{t^{1/12}(CQ)^{1/4}}{(\Delta_1\Delta_2\Delta_3)^{1/4}K^{1/2}}\frac{C^{1/8}N^{1/4}}{K^{1/4}Q^{5/8}t^{1/12}}\left(1+\frac{t^{1/12}N^{1/4}}{K}\right)\left(1+\frac{(t\Delta)^{1/2}}{N^{1/2}}\right),
\end{align}
in place of \eqref{in-place}.
This is dominated by $t^{-\delta}$ for $\delta<1/126$, if either $N\ll t^{1-10\delta}$ or $\Delta_1\Delta_2\Delta_3\gg t^{18\delta}$.
\end{proof}

\begin{Lemma}
\label{lemma:counting-problem-2}
Suppose $t^{1-10\delta}\ll N\ll t^{1+\varepsilon}$, and $\Delta_1\Delta_2\Delta_3\ll t^{18\delta}$. 
The contribution of $u\neq 0$ to the dual off-diagonal is satisfactory if $\delta<1/600$ and $M_2\Delta_2\gg t^{36\delta}$.
\end{Lemma} 

\begin{proof}
The first condition of \eqref{counting-problem-final} implies that 
$$
\left\| \frac{h_2c_1}{c_2}\right\|\ll \frac{Ct\Delta}{N}\ll \frac{Qt^{2/3}}{N}.
$$
From this we conclude that
$$
\left\| \frac{h_2d_1\delta_1}{d_2}\right\|\ll \frac{Qt^{2/3}\Delta_2}{N}.
$$
Now $h_2\equiv d_2'\bar{\mu}_2\bmod{d_2}$, and $d_2'\equiv -u\bar{d}_1\bmod{d_2}$. So the condition reduces to
$$
\left\| \frac{u\bar{\mu}_2\delta_1}{d_2}\right\|\ll \frac{Qt^{2/3}\Delta_2}{N}.
$$
We detect this event using the exponential sum
\begin{align*}
\frac{1}{\Xi}\left|\sum_{\alpha\sim \Xi} e\left(\frac{\alpha u\bar{\mu}_2\delta_1}{d_2}\right)\right|
\end{align*}
where $\Xi=N/Qt^{2/3}\Delta_2$. (One would notice that in the ranges under consideration $\Xi>1$.) Applying the reciprocity relation we get
\begin{align*}
\frac{1}{\Xi}\left|\sum_{\alpha\sim \Xi} e\left(-\frac{\alpha u\bar{d}_2\delta_1}{\mu_2}\right)e\left(\frac{\alpha u\delta_1}{\mu_2 d_2}\right)\right|.
\end{align*}
Next consider the sum 
\begin{align*}
\sum_{d_2\sim C/\Delta_2\Delta_3}\:\left|\sum_{\alpha\sim \Xi} e\left(-\frac{\alpha u\bar{d}_2\delta_1}{\mu_2}\right)e\left(\frac{\alpha u\delta_1}{\mu_2 d_2}\right)\right|^2
\end{align*}
which is trivially bounded by $C\Xi^2/\Delta_2\Delta_3$.
Opening the absolute square and applying the Poisson summation formula on the sum over $d_2$ we arrive at 
\begin{align*}
\frac{C}{\Delta_2\Delta_3\mu_2} \sum_{d_2\in \mathbb{Z}}\:\mathop{\sum\sum}_{\alpha_1,\alpha_2 \sim \Xi} \;S(d_2,-\xi u\delta_1; \mu_2) \int W(y) e\left(\frac{\xi  u\delta_1\Delta_2\Delta_3}{\mu_2 Cy}-\frac{Cd_2y}{\Delta_2\Delta_3\mu_2}\right)\mathrm{d}y,
\end{align*}
where $\xi=\alpha_1-\alpha_2$. The case $\xi=0$ is treated separately. In this case we get the diagonal contribution with a saving of $\Xi$. It follows, using the Weil bound for the Kloosterman sums, that the contribution of $\xi\neq 0$ is bounded by
\begin{align*}
\frac{C\Xi }{\Delta_2\Delta_3\mu_2} \:\mathop{\sum}_{1\leq \xi  \ll \Xi} \;(\xi u\delta_1,\mu_2)+
\frac{C\Xi }{\Delta_2\Delta_3\mu_2^{1/2}} \:\mathop{\sum}_{1\leq \xi  \ll \Xi} \;(\xi u\delta_1,\mu_2)^{1/2} \;\sum_{1\leq |d_2|\ll D} 1
\end{align*}
where 
$$
D=\frac{\xi u\Delta_1(\Delta_2\Delta_3)^2}{C^2}+\frac{\Delta_2\Delta_3\mu_2}{C}.
$$
Since $\sum_{\mu_2\sim M_2/\Delta_2} (\xi u\delta_1,\mu_2)\ll t^\varepsilon M_2/\Delta_2$, it follows that in the first term we save $M_2/\Delta_2$ on average over $\mu_2$. For the second term we observe that $\xi u \Delta_1(\Delta_2\Delta_3)^2/C^2$ is smaller than $N\Delta_3\Delta/QK^2$, which is $O(t^{-\varepsilon})$ if $\delta\leq 1/300$. The second term reduces to (on average over $\mu_2$) 
$$
\frac{\Xi^2 M_2^{1/2}}{\Delta_2^{1/2}}\ll \frac{\Xi^2 (NC\ell)^{1/2}}{\Delta_2^{1/2}(QK^2)^{1/2}}.
$$
Hence we have saved $(CQ)^{1/2}K/N^{1/2}\ell^{1/2}\Delta_2^{1/2}\Delta_3$. Consequently in place of \eqref{generic-count} we get
\begin{align*}
    \frac{C^{3/8}N^{1/4}}{(\Delta_1\Delta_2\Delta_3)^{1/4}K^{3/4}Q^{3/8}}\:\left(1+\frac{t^{1/12}N^{1/4}}{K}\right)\left(\frac{1}{\Xi^{1/8}}+\frac{\Delta_2^{1/8}}{M_2^{1/8}}+\frac{(N\Delta_2)^{1/16}\Delta_3^{1/8}}{(CQ)^{1/16}K^{1/8}}\right).
\end{align*}
The contribution of the middle term in the second pair of braces is $O(t^{-\delta})$ if $M_2\Delta_2\gg t^{36\delta}$. The contribution of the other terms turns out to be dominated by
$$
t^{7\delta/2}(t^{-1/120+7\delta/20}+t^{-1/80-\delta/10})\ll t^{-1/120+77\delta/20},
$$
which is $O(t^{-\delta})$ if $\delta<1/600$.
\end{proof}

\begin{remark}
The above lemma takes care of the generic case. It is important to observe that the proof exploits the Weil bound for Kloosterman sums. We are now left with the special case where $M_2$ is small. In this case we will use a non-trivial (though elementary) result from \cite{BI}.
\end{remark}

\begin{Lemma}
\label{lemma:counting-problem-3}
Suppose $t^{1-10\delta}\ll N\ll t^{1+\varepsilon}$, and $\Delta_1\Delta_2\Delta_3\ll t^{18\delta}$. 
The contribution of $u\neq 0$ to the dual off-diagonal is satisfactory if $\delta<1/1200$ and $M_2\Delta_2\ll t^{36\delta}$.
\end{Lemma} 

\begin{proof}
Consider the case of small $\mu_2\sim M_2/\Delta_2$. Then we need to count the number of solutions to 
$$
\left\|\frac{\bar{h}_2}{c_2}-\frac{\bar{h}_2'}{c_2'}\right\|\leq \frac{\mu_2\Delta_2}{C^2},\;\;\;\text{and}\;\;\; |h_2c_2-h_2'c_2'|\leq \frac{tC^2}{N}\left(\frac{M_2QK^2}{NC\ell}+\Delta\right).
$$ 
The last inequality comes from \eqref{hc-distance}.
Let $j=(h_2,c_2')$ and $k=(h_2',c_2)$. First we will show that $jk$ can not be large. We write $h_2=j\mathrm{h}$, $c_2'=j\mathrm{c}'$, $h_2'=k\mathrm{h}'$ and $c_2=k\mathrm{c}$. Then the congruence in \eqref{counting-problem-final} reduces to
$$
\bar{\mathrm{h}}\mathrm{c}'-\bar{\mathrm{h}}'\mathrm{c}\equiv m_2\bmod{c_2c_2'},\;\;\;\text{or}\;\;\;\mathrm{h}'\mathrm{c}'-\mathrm{h}\mathrm{c}\equiv \mathrm{h}\mathrm{h}'m_2\bmod{c_2c_2'}.
$$
Now $\mathrm{h}\mathrm{c}\asymp \mathrm{h}'\mathrm{c}'\asymp tC^2/Njk$, and $\mathrm{h}\mathrm{h}'m_2\asymp t^2C^2M_2/N^2jk$. So if $jk\gg t^{56\delta}/\Delta_2$, then all the terms in the congruence have size smaller than $C^2$, and consequently the congruence reduces to an equation
$$
\mathrm{h}'\mathrm{c}'-\mathrm{h}\mathrm{c}= \mathrm{h}\mathrm{h}'m_2.
$$
It follows that
$$
|h_2c_2-h_2'c_2'|=jk|\mathrm{h}\mathrm{h}'m_2|\leq \frac{tC^2}{N}\left(\frac{M_2QK^2}{NC\ell}+\Delta\right),
$$ 
or
$$
1\leq \frac{QK^2}{tC\ell}+\frac{N\Delta}{tM_2}.
$$ 
But this is not possible, and hence we obtain the bound $jk\gg t^{56\delta}/\Delta_2$. \\

We now observe that 
$$
\frac{M_2QK^2}{NC\ell}+\Delta\ll \frac{M_2N^{1/2}}{Qt^{1/3-\delta}}+t^{-1/3}\ll \frac{M_2t^{-9\delta/5}}{N^{1/10}}+t^{-1/3} \ll t^{-1/10}M_2. 
$$ 
The calculations from Section~4 (in particular Lemmas~4.1 and 4.2) of \cite{BI} can now be readily used. We can not use the main theorem of that section, as we need to avoid the diagonal. We observe that cases~1 and 2, and the diagonal contributions in cases~3 and 4 of Lemma~4.1 of \cite{BI}, are not present in our count. Indeed in our case $jk\ll t^{56\delta}<Ct^{-\varepsilon}$ and so $\mathfrak{a}_1=\mathfrak{c}=1$ in the notation of Lemma~4.1 of \cite{BI} can not occur. Also the diagonal case $\mathfrak{a}_1=\mathfrak{a}$, $\mathfrak{c}_1=\mathfrak{c}$ in the notation of \cite{BI}, is ruled out as this in our case leads to 
$$
\bar{\mathrm{h}}\mathrm{c}'-\bar{\mathrm{h}}'\mathrm{c}\equiv 0\bmod{c_2c_2'},
$$
i.e. $m_2=0$ (as $m_2\sim M_2\ll t^{36\delta}\ll C^2t^{-\varepsilon}$). It follows that the number of $(\mathbf{v}_2,m_2)$ vectors is bounded by
\begin{align*}
\frac{t^{56\delta+\varepsilon}}{\Delta_2}\left\{M_2t^{-1/10}\frac{t^2C^2}{N^2}+t^{-1/10+36\delta}\frac{M_2}{C^2}\frac{t^2C^4}{N^2}+ \frac{M_2^2}{C^4}\frac{t^2C^4}{N^2}\right\}\ll M_2t^{92\delta-1/10+\varepsilon}\frac{t^2C^2}{\Delta_2 N^2}.
\end{align*}
Then given $u\neq 0$, we determine $c_1, c_1'$ from the equation $c_1'c_2-c_1c_2'=u$. Indeed $c_1$ is determined modulo $c_2$ by $c_1c_2'\equiv -u\bmod{c_2}$. Hence the number of $c_1$ is given by $O((c_2',c_2,u))=O((\delta_2,u))$. It follows that the number of $(c_1,c_1',u)$ is given by 
$$
\Delta_2\:\frac{C^3\ell t\Delta^2}{\Delta_1\Delta_3 QK^2}.
$$ 
Consequently in place of the bound given in Lemma~\ref{lemma:counting-problem} we get
\begin{align*}
\Delta_1  t^{92\delta-1/10+\varepsilon}M_2\frac{t^{7/3}C^3Q}{\ell \Delta_3 N^2K^2}\: \left(1+\frac{t^{1/3}N}{\Delta_1K^4}\right),
\end{align*}
as an upper bound for the count towards $N_{\neq 0}$. Then in place of \eqref{generic-count} we get
\begin{align*}
    \sqrt{N}HQ^2Kt^{1/3}\frac{(N\Delta_1)^{1/4}C^{3/8}}{\Delta_3^{1/4}K^{3/4}Q^{3/8}}\left(1+\frac{t^{1/12}N^{1/4}}{\Delta_1^{1/4}K}\right),
\end{align*}
which is seen to be satisfactory if $\delta<1/1200$.
\end{proof}

\bigskip
 We can now summarise the output of Sections~\ref{fe}--\ref{counting-problem} in form of a proposition.

\begin{proposition}
\label{prop-2}
Let $H=N/t^{1/3}\ll N^{1/2}t^{1/3-\delta}$. Suppose $Q$ and $K$ are chosen as in \eqref{pick-q} and \eqref{pick-k} respectively. Then
$$
\mathcal{F}\ll  \sqrt{N}HQ^2K t^{1/3-\delta}
$$ 
if $\delta<1/1200$.
\end{proposition}

\bigskip



\end{document}